\pdfoutput=1
\documentclass[11pt,reqno,oneside]{amsart}

\usepackage[letterpaper,hmargin=1.35in,vmargin=1.19in]{geometry}

\usepackage{amssymb}

\usepackage{tikz}
\usetikzlibrary{matrix,arrows,cd,calc}

\usepackage{mathtools} %
\usepackage{xspace}

\usepackage[hyphens]{url}
\usepackage{ifpdf}
\ifpdf  \usepackage[pdftex,bookmarks=false]{hyperref}
\else   \usepackage[dvips,bookmarks=false]{hyperref}
        \usepackage{breakurl} %
\fi
\usepackage{color}
\definecolor{darkgreen}{rgb}{0,0.45,0}
\hypersetup{colorlinks,urlcolor=blue,citecolor=darkgreen,linkcolor=darkgreen,linktocpage}

\numberwithin{equation}{section}

\usepackage[capitalize]{cleveref}

\usepackage[style=numeric,maxnames=10,backend=biber,hyperref=true]{biblatex}
\newtheorem{defn}{Definition}[section]
\newtheorem{lem}[defn]{Lemma}
\newtheorem{prop}[defn]{Proposition}
\newtheorem{cor}[defn]{Corollary}
\newtheorem{thm}[defn]{Theorem}
\newtheorem*{thm*}{Theorem}

\theoremstyle{remark}
\newtheorem{rmk}[defn]{Remark}
\newtheorem{eg}[defn]{Example}
\crefname{eg}{Example}{Examples}

\newcommand{\define}[1]{\textbf{\boldmath{#1}}}
\newcommand{\defeq}{\vcentcolon\equiv}  %

\def\noteson{\gdef\note##1{\marginpar{\raggedright\color{blue}$\cdot$##1}}}

\noteson

\ifpdf
  \usepackage{silence}
\fi

\newcommand{\loopspacesym}{\mathsf{\Omega}}
\newcommand{\loopspace}[1]{\loopspacesym(#1)}
\newcommand{\suspsym}{\mathsf{\Sigma}}
\newcommand{\susp}[1]{\suspsym #1}
\newcommand{\fib}[2]{{\mathsf{fib}}_{#1}(#2)}

\newcommand{\LLL}{\mathcal{L}}
\newcommand{\RRR}{\mathcal{R}}

\newcommand{\pto}{\to_{\ast}}
\newcommand{\degg}{\mathsf{deg}}
\newcommand{\sphere}[1]{\mathbb{S}^{#1}}
\newcommand{\base}{\ensuremath{\mathsf{base}}\xspace}
\newcommand{\lloop}{\ensuremath{\mathsf{loop}}\xspace}
\newcommand{\UU}{\mathcal{U}}
\newcommand{\Prop}{\mathsf{Prop}}

\newcommand{\refl}[1]{\ensuremath{\mathsf{refl}_{#1}}\xspace}
\newcommand{\transfib}[3]{\ensuremath{\mathsf{transport}^{#1}(#2,#3)\xspace}}
\newcommand{\emptyt}{\ensuremath{\mathbf{0}}\xspace}
\newcommand{\unit}{\ensuremath{\mathbf{1}}\xspace}
\newcommand{\colim}{\mathsf{colim}}

\newcommand{\Z}{\mathbb{Z}}
\newcommand{\N}{\mathbb{N}}
\newcommand{\Q}{\mathbb{Q}}

\newcommand{\blank}{\mathord{\hspace{1pt}\text{--}\hspace{1pt}}}
\DeclareMathOperator{\im}{im}
\newcommand{\idfunc}[1][]{\mathsf{id}_{#1}}
\newcommand{\ttrunc}[2]{\bigl\Vert #2\bigr\Vert_{#1}}
\newcommand{\ap}[1]{\mathsf{ap}_{#1}}
\newcommand{\trans}[2]{{#1}_{*}\!\left({#2}\right)}
\newcommand{\Hom}{\mathsf{Hom}}
\newcommand{\precomp}[1]{{#1}^*}
\newcommand{\bool}{\mathsf{bool}}
\newcommand{\true}{\mathsf{true}}

\newcommand{\pr}{\mathsf{pr}}
\newcommand{\trunc}[2]{\|#2\|_{{#1}}}
\newcommand{\tproj}[2]{|#2|_{{#1}}}

\newcommand{\lra}       {\longrightarrow}
\newcommand{\longhookrightarrow}{\lhook\joinrel\lra}
\newcommand{\sto}{\!\to\!}  %

\makeatletter
\newcommand{\ctsym}{%
  \mathchoice{\mathbin{\raisebox{0.5ex}{$\displaystyle\centerdot$}}}%
             {\mathbin{\raisebox{0.5ex}{$\centerdot$}}}%
             {\mathbin{\raisebox{0.25ex}{$\scriptstyle\,\centerdot\,$}}}%
             {\mathbin{\raisebox{0.1ex}{$\scriptscriptstyle\,\centerdot\,$}}}
  }
\newcommand{\ct}[3][]{
  \@ifnextchar\bgroup
    {#2 \mathbin{\ctsym_{#1}} \ct[#1]{#3}}
    {#2 \mathbin{\ctsym_{#1}} #3}
  }

\def\smsym{\sum}
\newcommand{\@thesum}[1]{\smsym_{(#1)}}
\newcommand{\sm}[1]{\@ifnextchar\bgroup{\@sm{#1}\sm}{\@sm{#1}}}
\newcommand{\@sm}[1]{\mathchoice{{\textstyle\@thesum{#1}}}{\@thesum{#1}}{\@thesum{#1}}{\@thesum{#1}}}

\def\prdsym{\prod}
\newcommand{\@ifnextchar@starorbrace}[2]
  {\@ifnextchar*{#1}{\@ifnextchar\bgroup{#1}{#2}}}

\newcommand{\@theprd}[1]{\prdsym_{(#1)}}
\newcommand{\@theiprd}[1]{\prdsym_{\{#1\}}}

\newcommand{\prd}{\@ifnextchar*{\@iprd}{\@prd}}
\newcommand{\@prd}[1]
  {\@ifnextchar@starorbrace
    {\@tprd{#1}\prd}
    {\@tprd{#1}}}
\newcommand{\@tprd}[1]{%
  \mathchoice{%
    {{\textstyle\@theprd{#1}}}}{\@theprd{#1}}{\@theprd{#1}}{\@theprd{#1}}}

\newcommand{\@iprd}[2]{\@ifnextchar@starorbrace%
  {\@tiprd{#2}\prd}%
  {\@tiprd{#2}}}
\newcommand{\@tiprd}[1]{
  \ifthenelse{\true}
    {\@@tiprd{#1}}%
    {\@tprd{#1}}}
\newcommand{\@@tiprd}[1]{\mathchoice{{\textstyle\@theiprd{#1}}}{\@theiprd{#1}}{\@theiprd{#1}}{\@theiprd{#1}}}
\newcommand{\eqvsym}{\simeq}    %
\newcommand{\eqv}[2]{
  \@ifnextchar\bgroup
    {#1 \eqvsym \eqv{#2}}
    {#1 \eqvsym #2}
  }

\def\lam#1{{\lambda}\@lamarg#1:\@endlamarg\@ifnextchar\bgroup{.\,\lam}{.\,}}
\def\@lamarg#1:#2\@endlamarg{\if\relax\detokenize{#2}\relax #1\else\@lamvar{\@lameatcolon#2},#1\@endlamvar\fi}
\def\@lamvar#1,#2\@endlamvar{(#2\,{:}\,#1)}
\def\@lameatcolon#1:{#1}

\def\lamu#1{{\lambda}\@lamuarg#1:\@endlamuarg\@ifnextchar\bgroup{.\,\lamu}{.\,}}
\def\@lamuarg#1:#2\@endlamuarg{#1}
\makeatother

\title{Localization in homotopy type theory}

\author[J.D. Christensen]{J. Daniel Christensen}
\address{University of Western Ontario, London, Ontario, Canada}
\email{jdc@uwo.ca}
\author[M. Opie]{Morgan Opie}
\address{Harvard University, Cambridge, Massachusetts, USA}
\email{opie@math.harvard.edu}
\author[E. Rijke]{Egbert Rijke}
\address{University of Illinois at Urbana-Champaign, Urbana, Illinois, USA}
\email{rijke@illinois.edu}
\author[L. Scoccola]{Luis Scoccola}
\address{University of Western Ontario, London, Ontario, Canada}
\email{lscoccol@uwo.ca}

\thanks{This material is based upon work supported by the National Science
Foundation under Grant Number DMS 1321794. Rijke gratefully acknowledges the support of the Air Force Office of Scientific Research through MURI grant FA9550-15-1-0053.
Opie is partially supported by an NSF Graduate Research Fellowship
through grant DGE-1144152.}
\thanks{Compared to the published version, this version has a correction to the
proof of \cref{Lseparatedloopspace}.}


\date{February 9, 2020}

\begin{document}

\begin{abstract}
We study localization at a prime in homotopy type theory, using self maps of the circle.
Our main result is that for a pointed, simply connected type $X$,
the natural map $X \to X_{(p)}$ induces algebraic localizations on
all homotopy groups.
In order to prove this, we further develop the theory of reflective subuniverses.
In particular, we show that for any reflective subuniverse $L$,
the subuniverse of $L$-separated types is again a reflective subuniverse,
which we call $L'$.
Furthermore, we prove results establishing that $L'$ is almost left exact.
We next focus on localization with respect to a map,
giving results on preservation of coproducts and connectivity.
We also study how such localizations interact with other reflective subuniverses and orthogonal factorization systems.
As key steps towards proving the main theorem, we show that localization at a prime
commutes with taking loop spaces for a pointed, simply connected type,
and explicitly describe the localization of an
Eilenberg-Mac Lane space $K(G,n)$ with $G$ abelian.
We also include a partial converse to the main theorem.
\end{abstract}

\maketitle

\tableofcontents

\section{Introduction}

Problems in topology and algebra often become simpler when localized at
a prime $p$, that is, when all other primes are in a sense invertible.
In this paper, we study such localizations in homotopy type theory
and develop the necessary general theory in order to prove results
about localization at a prime.
Writing $X_{(p)}$ for the localization of a type $X$ at a prime $p$,
a special case of our main theorem (\cref{theorem:localizationlocalizes})
says:

\begin{thm*}
Let $X$ be a pointed, simply connected type and let $p$ be a prime.
Then for each $n \geq 2$, the maps $\pi_n(X) \to \pi_n(X_{(p)})$
induced by the natural map $X \to X_{(p)}$ are algebraic localizations
of abelian groups at $p$.
\end{thm*}

In order to prove this theorem, it is necessary to develop a
substantial number of results, many of which are of independent interest.
As we are working in homotopy type theory, which has models%
\footnote{This depends, as usual, on the initiality conjecture for homotopy type theory.
We also assume that the universe is closed under the needed higher inductive types.}
in all $\infty$-toposes (\cite{KL}, \cite{LS}, \cite{S}),
our results hold in more general situations than the homotopy theory of spaces.
To achieve this greater generality,
all of our proofs and constructions must be homotopy invariant,
all of our arguments must be constructive (avoiding the law of excluded middle
and the axiom of choice), and we cannot make use of Whitehead's theorem,
which does not hold in this generality.
This means that many of our proofs are new, and do not follow sources
such as~\cite{MayPonto}.
Our spirit is closer to that of~\cite{DrorFarjoun}, but still our methods
are different in key places.
For example, our proof that $\loopspacesym L_{\susp{f}} X \simeq L_f \loopspacesym X$
does not follow the approach used in~\cite{DrorFarjoun}, as that argument
relies on delooping machinery that we do not have available.
Our proof instead uses a univalent universe.
Note that, based on the present paper, Marco Vergura has directly proved many
results from \cref{section:rsst} in an arbitrary $\infty$-topos (\cite{Vergura1},
\cite{Vergura2}).

\medskip

We will now explain the ingredients that lead to the main theorem.

We begin in \cref{section:rsst} by studying general reflective subuniverses,
building on the work of~\cite{RSS}.
A \define{reflective subuniverse} $L$, which we also call a \define{localization},
is a predicate specifying which types are $L$-local along with a
\define{unit} map $\eta : X \to LX$, for each $X$,
which is initial among maps whose codomain is local.
For example, given a map $f : A \to B$, we say that a type $X$ is \define{$f$-local}
if precomposition with $f$ gives an equivalence $(B \to X) \to (A \to X)$.
In~\cite{RSS} it is shown using higher inductive types that the $f$-local types
form a reflective subuniverse $L_f$.

We prove a variety of results about general reflective
subuniverses.  We highlight here our discussion of the $L$-separated types.
Given a reflective subuniverse $L$, a type $X$ is \define{$L$-separated} if
its identity types are $L$-local.
Using the join construction~\cite{joinconstruction},
we prove that the subuniverse of $L$-separated types is reflective.
Using a dependent induction for $L'$, we prove that the map $X \to L' X$ induces $L$-localization on identity types.
In particular, when $X$ is pointed, $\loopspacesym L' X \simeq L \loopspacesym X$.
Our proof relies on a result that shows that $X \to L' X$ has a
constrained dependent elimination principle.
This principle also lets us show that $L$ and $L'$ together behave
similarly to a left exact (lex) modality.  For example, we give
results on the preservation of pullbacks and fiber sequences.

In~\cref{section:localization}, we study localization with respect to a
map $f$, or more generally, with respect to a family of maps.
In the case when $L$ is $L_f$, it turns out that $L'$ is $L_{\suspsym f}$.
This allows us to apply our results about $L'$-localization to this
situation, showing for example that $\loopspacesym L_{\susp{f}} X \simeq L_f \loopspacesym X$.
We prove that under certain conditions,
\pagebreak
$L_f X \simeq L_{\suspsym f} X$, which is a key tool for work in the next section.
This section also contains a number of results about the preservation of
coproducts and $n$-connectedness under appropriate localizations.
For example, if $\trunc{n}{f}$ is an equivalence and $X$ is $n$-connected,
then $L_f X$ is $n$-connected.
This result is in fact a corollary of a more abstract result concerning
the interaction between two reflective subuniverses.
In~\cref{ss:orthogonal-factorization-systems}, we strengthen this corollary
using a theorem which says that if $(\LLL, \RRR)$ is an orthogonal factorization system and
$f$ is a map in $\LLL$, then $\eta : X \to L_f X$ is in $\LLL$ for every $X$.

In~\cref{section:localizationaway}, we specialize to the study of localizations
at and away from primes.
In homotopy type theory, algebraic invariants such as homology or homotopy groups
are not strong enough to characterize these localizations.
Instead, following~\cite{CasacubertaPeschke}, we invert a prime $q$ by localizing at
the degree $q$ map $\degg(q) : \sphere{1} \to \sphere{1}$.
More generally, we localize at a family of such maps to invert a family of primes.
In particular, we can localize \emph{at} a prime $p$ by localizing with respect to the
family of maps $\degg(q)$ for all primes $q$ different from $p$.
Our results are stated for a general family $S$ of natural numbers,
and we write $\degg(S)$ for the family of degree $k$ maps for $k$ in $S$.%
\footnote{In some parts of \cref{section:localizationaway}, we require
the family $S$ to be indexed by $\N$.}

We begin by defining what it means for a group $G$ to be uniquely $S$-divisible
and show that this holds if and only if the associated Eilenberg-Mac Lane space
$K(G,n)$ is $\degg(S)$-local (where $G$ is assumed to be abelian if $n > 1$).
We also prove that every $\degg(S)$-local type has uniquely $S$-divisible homotopy groups
and show the converse for types that are simply connected and $n$-truncated for some $n$.
(We need the assumption that the type is truncated since Whitehead's theorem
does not hold in type theory.)

Using results from~\cref{section:localization}, we see that if $X$ is $n$-connected,
its $\degg(S)$-localization agrees with its $\suspsym^n \degg(S)$-localization,
where $\suspsym^n \degg(S)$ is the family of $n$-th suspensions of the maps
in the family $\degg(S)$.
In particular, when $X$ is pointed and simply connected, we deduce that
$\loopspacesym L_{\degg(S)} X \simeq L_{\degg(S)} \loopspacesym X$.
We use this to show that for $X$ connected,
the map $X \to L_{\degg(S)} X$ induces an algebraic localization of the first
non-trivial homotopy group of $X$.
In particular, by taking $X$ to be an Eilenberg-Mac Lane space,
this implies that the algebraic localization of any group exists.

We next give an explicit description of the $\degg(S)$-localization of
an Eilenberg-\break Mac~Lane space $K(G,n)$ for $G$ abelian,
and deduce that the result is an Eilenberg-Mac~Lane space for the
localized group.
This has the indirect consequence that localizing an abelian group $G$
among non-abelian groups is the same as localizing $G$ among abelian groups.
Since we were unable to find this statement in the literature, we give an
independent, elementary proof of this in~\cref{ss:abelian}, which is not
needed elsewhere.

With all of this in place, we prove the main result: $\degg(S)$-localization of a pointed, simply connected type $X$
localizes all of the homotopy groups of $X$ away from $S$.
We also prove a partial converse.  Let $f : X \to X'$ be a pointed
map between pointed, simply connected $n$-types for some $n$.
If $f$ induces algebraic localization away from $S$ in $\pi_m$
for each $m > 1$, then $f$ is a $\degg(S)$-localization of $X$.

\medskip

We rely heavily on~\cite{RSS}.  That paper focuses on
modalities, and we show in~\cref{ss:localizationofgroups} that our
motivating examples are not modalities.
This is why we spend time on general reflective subuniverses before
specializing to $\degg(S)$-localization in the final section.
The constrained dependent elimination principle for $L'$-localization
shows that $L'$ is close to being a modality, which turns out to be sufficient
for our purposes.
We expect that these general results will find other applications as well.

\medskip

Those familiar with the classical story will notice some omissions.
First, we have restricted to simply connected types in many places,
where the classical approach applies to nilpotent spaces.
Secondly, we do not prove fracture theorems which reconstruct a type
from its localizations.
Nilpotent types and fracture squares in homotopy type theory are developed in~\cite{scoccola}.
This is a natural point to mention that very little of the present
work has been formalized in a proof assistant, but that we hope to
address this in the future.%
\footnote{Since this paper was published, Mike Shulman formalized many
of the results from \cref{section:rsst} in the HoTT Library:
\url{https://github.com/HoTT/HoTT}}
 
\medskip

\noindent
\textbf{Acknowledgements:}
We thank Marco Vergura and Alexandra Yarosh for helpful conversations about this project
and for comments during the writing of this paper.
We thank the referee for many suggestions that improved the paper.
We thank the American Mathematical Society 
for running the Mathematics Research Communities Program
in June, 2017, at which this work began,
and the National Science Foundation for supporting the MRC program.

\subsection{Notation and conventions}

We primarily follow the notation used in~\cite{hottbook}.
For example, we write $\emptyt$ for the empty type and $\unit$ for the unit type.

We also make use of standard background from~\cite{hottbook}, such as the notion of
a pointed type $(X, x_0)$ and its loop space $\loopspacesym (X, x_0) \defeq (x_0 = x_0)$,
which is written $\loopspacesym X$ when the basepoint can be suppressed.
A pointed map from $(X, x_0)$ to $(Y, y_0)$ is a map $f:X \to Y$ and a path $p: f(x_0) = y_0$.
We write $X \pto Y$ for the pointed type of pointed maps from $(X,x_0)$ to $(Y,y_0)$,
leaving the base points implicit.
A pointed map $f$ induces a pointed map $\loopspacesym X \to \loopspacesym Y$ in
a natural way.

We assume the existence of higher inductive types, which are used in~\cite{RSS}
to construct the localization with respect to a map.
We also work with the circle $\sphere{1}$ defined as a higher inductive type
with one point constructor $\base : \sphere{1}$ and one path constructor
$\lloop : \base = \base$.
Mapping from $\sphere{1}$ to a type $X$ is the same as providing an element $x : X$
and a loop $p : x = x$.
It also follows from the existence of higher inductive types that truncations exist.

We assume univalence and, in particular, function extensionality.
We write $\UU$ for a univalent universe, and assume that all free type variables lie in $\UU$.
In a few places we will deal with types not in $\UU$. In those cases a \define{small type} will
mean a type in $\UU$, and a \define{locally small type} will mean a type whose identity types
are equivalent to small types.

\section{Reflective subuniverses and their separated types}\label{section:rsst}

In this section, we develop the general theory of reflective subuniverses, drawing on~\cite{RSS} and emphasizing those properties that are necessary in what follows. 

We begin in \cref{subsection:reflectivesubuniverses} with definitions and preliminary observations. While we later specialize to $\degg(k)$-localization, working in greater generality clarifies the structure of many of the arguments. For example, other reflective subuniverses, such as the subuniverse of $n$-truncated types, naturally arise as we investigate $\degg(k)$-local types.

In \cref{ss:separated-types}, we focus on the separated types of a given reflective subuniverse $L$. These are the types whose identity types are in the subuniverse. In the case of localization with respect to a map $f$, the separated types are precisely the $\Sigma f$-local types. Many of the results that we will need for $\Sigma f$-localization can be phrased as more general results on $L$-separated types, and we prove them as such. 

Write $L'$ for the subuniverse of $L$-separated types.
\cref{ss:constructionofseparated} contains a proof that $L$-separated types form a reflective subuniverse. This is not necessary for our later results since, in the case of localization with respect to a map $f$, $L_f'$-local types and $L_{\Sigma f}$-local types coincide. However, this result may be of use to the reader interested in more general localizations.

In \cref{ss:lex}, we show that $L$ and $L'$ together behave similarly to a lex modality.
In particular, we characterize the identity types of $L'$-localizations and
give results about the preservation of pullbacks and fiber sequences.

\subsection{Reflective subuniverses}\label{subsection:reflectivesubuniverses}

In this section, we develop background on reflective subuniverses, building on~\cite{RSS}. Our investigation of localization with respect to families of maps, carried out in \cref{section:localization}, fits into this general framework.

\begin{defn}
    A \define{subuniverse} $L$ is a family $\mathsf{isLocal_L}:\UU \to\Prop$.
    Given $X : \UU$, if the type $\mathsf{isLocal_L}(X)$ is inhabited 
    we say that $X$ is \define{$L$-local}.
    We write $\UU_L \defeq \sm{X:\UU} \mathsf{isLocal_L}(X)$ for the subuniverse of
    $L$-local types.
\end{defn}

\begin{eg}\label{example:localizationatmaps}
    For any $n\geq -2$, being $n$-connected is a mere proposition, so the class of $n$-connected types
    forms a subuniverse.
\end{eg}

\begin{defn}
Given a subuniverse $L$ and a type $X$, an \define{$L$-localization} of $X$ consists of an $L$-local type
$X'$ and a map $g : X \to X'$ such that for every $L$-local type $Y$ the map
\[
  \precomp{g} : (X'\to Y) \lra (X\to Y)
\]
is an equivalence. We call this last fact \define{the universal property of $L$-localization}.
\end{defn}

A straightforward application of the universal property and the univalence axiom
shows that localizations are unique when they exist:

\begin{lem}[{\cite[Lemma 1.17]{RSS}}]
Given a subuniverse $L$, the type of $L$-localizations of $X$ is a mere proposition.\qed
\end{lem}

\begin{defn}
A \define{reflective subuniverse} $L$ consists of a subuniverse $\mathsf{isLocal_L}:\UU\to\Prop$,
a \define{reflector} $L:\UU\to\UU_L$
and a \define{unit}
\[
  \eta : \prd{X:\UU}X \to LX
\]
such that for every $X : \UU$, the
map $\eta_X : X \to LX$ is an $L$-localization of $X$.
\end{defn}

\begin{eg}\label{example:subuniversemaps}
Many examples of reflective subuniverses are obtained by \emph{localizing} at a family of maps $f:\prd{i:I}A_i\to B_i$.
In this context, a type $X$ is $f$-local if the map $\precomp{f_i} : (B_i \to X) \to (A_i \to X)$ is an equivalence
for all $i : I$.
By~\cite[Theorem 2.18]{RSS}, the $f$-local types form a reflective subuniverse which we denote by $L_f$.
Examples of this include $n$-truncation for any $n \geq -2$ and $\degg(k)$-localization.
We specialize to $L_f$ in \cref{section:localization} and specialize further
to inverting natural numbers in \cref{section:localizationaway}.
\end{eg}

In the rest of this section, $L$ will denote an arbitrary reflective subuniverse.
We recall the basic properties of reflective subuniverses from~\cite{RSS} and~\cite[Section~7.7]{hottbook}.
First, two reflective subuniverses with the same local types necessarily have the same reflector and the same unit.
This means that being reflective is a mere property of a subuniverse.
Moreover, a type $X$ is local if and only if the unit $\eta_X : X \to LX$ is an equivalence.
The reflector $L$ is automatically functorial in the sense that for any
$g : X \to Y$, there is a unique map, denoted $Lg : LX \to LY$, together with a homotopy making the following square commute
\[
        \begin{tikzpicture}
          \matrix (m) [matrix of math nodes,row sep=2em,column sep=3em,minimum width=2em]
          { X & L X \\
            Y & L Y . \\};
          \path[->]
            (m-1-1) edge node [above] {$\eta$} (m-1-2)
                    edge node [left] {$g$} (m-2-1)
            (m-2-1) edge node [above] {$\eta$} (m-2-2)
            (m-1-2) edge [dashed] node [right]{$L g$} (m-2-2)
            ;
        \end{tikzpicture}
\]

Any reflective subuniverse contains the unit type and is closed under
pullbacks, products, identity types, dependent products over any type,
sequential limits and other
limits that can be defined in homotopy type theory~\cite{AKL}.
Note, however, that reflective subuniverses are not closed under dependent sums in general,
even if the indexing type is in the reflective subuniverse.
(See \cref{example:nonlocalfib2,example:nonlocalfib}.)

The universal property of $L$-localization can be regarded as a recursion principle.
From this point of view, it turns out that a reflective subuniverse has an induction principle (dependent elimination)
precisely when it is closed under dependent sums:

\begin{thm}[{\cite[Theorem 1.32]{RSS}}]\label{modality}
The following are equivalent:
\begin{enumerate}
\item For any local type $X$ and any family $P:X \to \UU_L$ of local types, the dependent pair type $\sm{x:X} P(x)$ is again local.
\item For any type $X$ and any family $P:LX\to\UU_L$, the precomposition map
\[
  \precomp{\eta_X}: \Big( \prd{l:LX} P(l) \Big) \lra \Big( \prd{x:X} P(\eta(x)) \Big)
\]
is an equivalence.  %
\item For any type $X$, the unit $\eta_X: X \to LX$ is \define{$L$-connected}, i.e., for any $l:LX$, the localization $L(\fib{\eta_X}{l})$ is contractible.
\end{enumerate}
If any of these equivalent conditions hold, we say that $L$ is a \define{modality}.\qed
\end{thm}

Although an arbitrary reflective subuniverse need not have dependent elimination into
families of local types (property (2), above), we observe that if property (1) holds for a particular type family over $LX$,
then property (2) holds for that type family. This gives a restricted version of the dependent elimination principle,
which we will use several times in what follows to circumvent the fact that $\degg(k)$-localization is not a modality.

\begin{prop}\label{theorem:generalized-induction}
Consider a type $X$ and a type family $P:LX \to \UU$ such that the total space $\sm{l:LX} P(l)$ is local.
Then the precomposition map
\[
  \precomp{\eta_X} : \Big(\prd{l:LX}P(l)\Big) \lra \Big(\prd{x:X}P(\eta_X(x))\Big)
\]
is an equivalence.
\end{prop}

This follows from the proof that (1) implies (2) in~\cite[Theorem~1.32]{RSS}, but we include
a proof here for completeness.

\begin{proof}
Since $LX$ is local and $\sm{l:LX}P(l)$ is local, the precomposition maps $\precomp{\eta_X}$ in the commuting square
\[
  \begin{tikzcd}
    (LX\to \sm{l:LX}P(l)) \arrow[r,"\precomp{\eta_X}"] \arrow[d,swap,"\pr_1\circ\blank"] & (X\to \sm{l:LX}P(l)) \arrow[d,"\pr_1\circ\blank"] \\
    (LX\to LX) \arrow[r,swap,"\precomp{\eta_X}"] & (X\to LX)
  \end{tikzcd}
\]
are equivalences. It follows that they induce an equivalence from the fiber of the left-hand map $\pr_1\circ\blank$ at $\idfunc[LX]$ to the fiber of the right-hand map $\pr_1\circ\blank$ at $\eta_X$. In other words, we have an equivalence
\[
  \precomp{\eta_X} : \Big(\prd{l:LX}P(l)\Big) \lra \Big(\prd{x:X}P(\eta_X(x))\Big).\qedhere
\]
\end{proof}

\medskip
As one might expect, maps that become equivalences after $L$-localization
will become relevant later.
We call such a map an \define{$L$-equivalence}.

\begin{lem}\label{lemma:characterizationorthogonal}
For a map $g : X \to Y$, the following are equivalent:
\begin{enumerate}
\item $g$ is an $L$-equivalence.
\item For any local type $Z$, the precomposition map
\[
  \precomp{g} : (Y \to Z) \lra (X \to Z)
\]
is an equivalence.
\end{enumerate}
\end{lem}

This implies in particular that the unit $\eta : X \to L X$ is an $L$-equivalence for any type $X$.

\begin{proof} First we show that (1) implies (2).
    Let $Z$ be $L$-local. From the square used to define $Lg$,
    we can factor the map $(Y \to Z) \xrightarrow{\ \precomp{g}} (X \to Z)$ as
    \[
        (Y \to Z) \simeq (L Y \to Z) \xrightarrow{\precomp{(L g)}} (L X \to  Z) \simeq (X \to Z).
    \]
 Hence if $L g$ is an equivalence, then $\precomp{g} : (Y \to Z) \to (X \to Z)$ is an equivalence.
 
    Conversely, assume that $\precomp{g} : (Y \to Z) \to (X \to Z)$ is an equivalence for every $L$-local type $Z$.
    Then, using the same factorization and choosing $Z$ to be $L X$ and $L Y$, we deduce that $L g$ must be an equivalence.
\end{proof}

\cref{lemma:characterizationorthogonal} also implies that $L$-equivalences are closed under transfinite composition.
We make use of the notion of sequential colimit from \cite[Section 3.1]{Brunerie}.

\begin{lem}\label{lemma:orthogonalcomposition}
    If the maps in a sequential diagram
    \[X_0 \xrightarrow{\,h_0\,} X_1 \xrightarrow{\,h_1\,} X_2 \xrightarrow{\,h_2\,} \cdots\]
    are $L$-equivalences, then the transfinite composite $\overline{h} : X_0 \to \colim_n X_n$ is an $L$-equiva\-lence.
\end{lem}

\begin{proof}
    By \cref{lemma:characterizationorthogonal}, it is enough to check that
    $\precomp{\overline{h}} : (\colim_n X_n \to Z) \to (X_0 \to Z)$
    is an equivalence for every $L$-local type $Z$.
    By the induction principle of the sequential colimit,
    we can factor $\precomp{\overline{h}}$ as
    \[
        (\colim_n X_n \to Z) \xrightarrow{\;\sim\;} \lim_n (X_n \to Z) \lra (X_0 \to Z),
    \]
    and by hypothesis the transition maps $(X_{n+1} \to Z) \to (X_n \to Z)$ in the limit diagram are equivalences,
    so the second map is an equivalence as was to be shown.
\end{proof}

In \cref{section:localizationaway}, we will be interested in the effect of localization on the homotopy groups of a type.
For this, we need to understand the relationship between localization and truncation.
Since truncations are also examples of reflections onto reflective subuniverses,
the following two general lemmas will be useful.

The first lemma is a straightforward generalization of~\cite[Theorem~3.30]{RSS}.

\begin{lem}\label{lemma:commutelocalization}
    Let $K$ and $L$ be reflective subuniverses, and let $X$ be a type.
    If $KLX$ is $L$-local and $LKX$ is $K$-local, then $LKX = KLX$.\qed
\end{lem}

\begin{lem}\label{lemma:comparelocalization}
    Let $K$ and $L$ be reflective subuniverses with $K$ contained in $L$.
    Write $\eta^K$ and $\eta^L$ for the units.
    \begin{enumerate}
    \item If $f$ is an $L$-equivalence, then it is a $K$-equivalence.
          In particular, for any $X$, $K\eta^L_X : KX \to KLX$ is an equivalence.
    \item If $X$ is $K$-local, then $\eta^L_X : X \to LX$ is an equivalence.
          In particular, for any $X$, $\eta^L_{KX} : KX \to LKX$ is an equivalence.
    \item If $X$ is a type such that $LX$ is $K$-local, then the 
          natural map $LX \to KX$ is an equivalence.
    \end{enumerate}
\end{lem}

\begin{proof}
(1) follows from \cref{lemma:characterizationorthogonal}, and (2) is clear.
For (3), one checks that the unit $\eta^L_X : X \to LX$ has the universal
property of $K$-localization.
\end{proof}

\subsection{The subuniverse of separated types}\label{ss:separated-types}

We next investigate the types whose identity types are $L$-local.
We call these the $L$-separated types and denote the subuniverse by $L'$.
We show that the universe of $L$-local types is $L$-separated, up to size
issues, and this is sufficient to extend families of $L$-local types over
$L'$-localizations, an important tool in our work.
We finish with a constrained dependent elimination rule.

\begin{defn}
Let $L$ be a reflective subuniverse and let $X : \UU$ be a type. 
We say that $X$ is \define{$L$-separated} if its identity types are $L$-local types.
We write $L'$ for the subuniverse of $L$-separated types.
\end{defn}

In other words, a type $X$ is $L$-separated if its diagonal $\Delta:X\to X\times X$ is classified by $\UU_L$.

\begin{eg}\label{example:truncationisseparated}
Given $n \geq -2$, the subuniverse of $(n+1)$-truncated types is precisely the subuniverse of separated types for the reflective subuniverse of $n$-truncated types.
\end{eg}

More generally, for any family of maps $f$,
the separated types for the subuniverse of $f$-local types
are also characterized in a simple way: 

\begin{lem}\label{lemma:characterizationsigmaflocal}
    Let $f:\prd{i:I}A_i\to B_i$ be a family of maps. Denote the family consisting of the suspensions
    of the functions by $\susp{f} : \prd{i:I} \susp{A_i} \to \susp{B_i}$.
    Then, a type $X$ is $\suspsym f$-local if and only if it is $f$-separated.
    In other words, $L_{\susp{f}} = (L_{\!f})'$.
\end{lem}

\begin{proof}
    By the induction principle for suspension and naturality, we obtain for each $i : I$ a commutative square
\[
  \begin{tikzcd}
    (\susp{B_i} \to X) \arrow[r,"\simeq"] \arrow[d] & \left( \sm{x,y:X} (B_i \to x = y) \right) \arrow[d] \\
    (\susp{A_i} \to X) \arrow[r,"\simeq"] & \left( \sm{x,y:X} (A_i \to x = y) \right)
  \end{tikzcd}
\]
in which the horizontal maps are equivalences.
So $X$ is $\suspsym f$-local if and only if the right vertical map is an equivalence
for every $i : I$, if and only if for each $x,y : X$, the type $x = y$ is $f_i$-local
for every $i : I$.
\end{proof}

Notice that in this case the subuniverse of separated types is reflective
since it is again localization with respect to a family of maps~\cite[Theorem~2.18]{RSS}.
This holds more generally: we will prove in \cref{ss:constructionofseparated} that the subuniverse
of separated types is always reflective.
In the remainder of this section, we give results that will be used for that proof
as well as in later parts of the paper.

\begin{rmk}
The following are true for any reflective subuniverse $L$.
\begin{enumerate}
\item Since $L$-local types are closed under pullback, it follows that all $L$-local types are $L$-separated.
\item Since pullbacks commute with identity types, it follows that $L$-separated types are closed under pullbacks. Similarly, $L$-separated types are closed under retracts, subtypes, and dependent products of families of $L$-separated types indexed by an arbitrary type.
\item Since the unit type is $L$-local, it follows that every mere proposition is $L$-separated.
\item If $L$ is closed under dependent sums, then $L'$ is also closed under dependent sums,
    by the characterization of identity types of dependent sums~\cite[Theorem~2.7.2]{hottbook}.
    So, given that separated types form a reflective subuniverse, it will follow that if $L$ is a modality,
    then so is $L'$.
\end{enumerate}
\end{rmk}

\begin{lem}\label{lemma:etasurjective}
    The unit of an $L'$-localization $\eta' : X \to L'X$ is surjective.
\end{lem}
\begin{proof}
By definition, a type is $L$-separated if and only if its identity types are $L$-local.
So any subtype of an $L$-separated type is again $L$-separated.
It follows that the image of $\eta'$ is $L$-separated, and thus we have an extension
\[
  \begin{tikzcd}
    X \arrow[r,"\tilde{\eta'}"] \arrow[d,swap,"\eta'"] & \im(\eta') \\
    L'X \arrow[ur,dashed]
  \end{tikzcd}
\]
which is easily seen to be a section of the inclusion $\im(\eta')\hookrightarrow L' X$.
Therefore, $\eta'$ is surjective.
\end{proof}

\begin{prop}\label{prop:UU_L-is-L-separated}
The identity types of the subuniverse $\UU_L$ are equivalent to $L$-local types.
\end{prop}

\begin{proof}
    Note that for any two $L$-local types $A$ and $B$ we have $\eqv{(A=_{\UU_L}B)}{(\eqv{A}{B})}$ by univalence and the
    fact that being $L$-local is a mere proposition.
    Now, notice that the type $\eqv{A}{B}$ is equivalent to the pullback
  \[
    \begin{tikzcd}[column sep=huge]
      (\eqv{A}{B}) \arrow[r] \arrow[d] & \unit \arrow[d,"{(\idfunc[A],\idfunc[B])}"] \\
      (A \sto B) \times (B \sto A) \times (B \sto A)
         \arrow[r,swap,"{(f,g,h) \longmapsto (hf,fg)}" yshift=-1ex] & (A \sto A) \times (B \sto B).
    \end{tikzcd}
  \]
  of $L$-local types, so it is $L$-local.
\end{proof}

The only thing that prevents $\UU_L$ from actually being $L$-separated is the fact that
$\UU_L$ is not small.
But we can still treat it as an $L$-separated type in the following sense.

\begin{lem}\label{lemma:extendtoUL}
Every type family $P : X \to \UU_L$ extends uniquely along any $L'$-localization $X \to L'X$.
That is, every map $Y \to X$ with $L$-local fibers is the pullback along $X \to L'X$ of
a unique map $Y' \to L'X$ with $L$-local fibers.
\end{lem}

\begin{proof}
    We prove the first form of the statement.
    By \cref{prop:UU_L-is-L-separated}, the identity types of $\UU_L$ are
    equivalent to small types, i.e., $\UU_L$ is a locally small type.
    By the join construction~\cite{joinconstruction}, the image of $P$
    can be taken to be a small type $I$ in $\UU$, so there is a factorization
    of $P$ into a surjection $\hat{P} : X \to I$ followed by
    an embedding $i : I \to \UU_L$:
    \[
        \begin{tikzpicture}
          \matrix (m) [matrix of math nodes,row sep=2em,column sep=3em,minimum width=2em]
          { X & \UU_L \\
            L' X & I. \\};
          \path[->]
            (m-1-1) edge node [above] {$P$} (m-1-2)
                    edge node [left] {} (m-2-1)
                    edge node [above] {$\hat{P}$} (m-2-2)
            (m-2-2) edge node [right] {$i$} (m-1-2)
            ;
        \end{tikzpicture}
    \]
    Since the identity types of $I$ are equivalent to identity types
    of $\UU_L$, and $I$ is small, it follows that the identity types of $I$ are actually $L$-local.
    This means that $I$ is an $L$-separated type, so we can extend $\hat{P}$ to $L' X$ 
    giving us the desired extension of $P$ by composing with $i$.

    Since $X \to L'X$ is surjective (\cref{lemma:etasurjective}),
    any such extension must factor through the image $I$.
    So uniqueness follows from the universal property of $L'$-localization.
\end{proof}

\begin{rmk} 
    In the special case in which $L$ is localization with respect to a
    family of maps $f$ (see \cref{section:localization}),
    which is enough for our purposes, \cref{lemma:extendtoUL} has a simpler proof.
    As \cref{lemma:characterizationsigmaflocal} shows,
    $L'$ corresponds to localization with respect to the family $\susp{f}$
    of suspended maps.
    In this case, the localization $L'$ can be constructed using a higher inductive type
    that can eliminate into any $f$-local type, not only small ones, as \cite[Lemma 2.17]{RSS} shows, so the result follows.
\end{rmk}

\begin{lem}\label{lemma:separatedpluslocalisseparated}
If $X$ is an $L$-separated type and $P:X\to \UU_L$ is a family of $L$-local types, then the type
$\sm{x:X}P(x)$ is $L$-separated.
\end{lem}

\begin{proof}
For any $(x,p)$ and $(y,q)$ in $\sm{x:X}P(x)$, the type $(x,p)=(y,q)$ is equivalent to the pullback
\[
  \begin{tikzcd}[column sep=huge]
    (x,p)=(y,q) \arrow[r] \arrow[d] & \unit \arrow[d,"q"] \\
    (x=y) \arrow[r,swap,"\transfib{P}{-}{p}"] & P(y)
  \end{tikzcd}
\]
of $L$-local types, so it is $L$-local. 
\end{proof}

\cref{lemma:separatedpluslocalisseparated} and \cref{theorem:generalized-induction}
imply that $L'$-localizations have a dependent elimination principle.

\begin{prop}\label{proposition:inductionLseparated}
Let $P:L'X \to \UU$ be a type family with $L$-local fibers.
Then precomposition with an $L'$-localization $\eta' : X \to L'X$ induces an equivalence
\[
    \prd{x : L'X} P(x) \simeq \prd{x : X} P(\eta' x). \tag*{\qed}
\]
\end{prop}

\subsection{Construction of $L'$-localization}\label{ss:constructionofseparated}

We next show that for any reflective subuniverse $L$, the subuniverse
$L'$ of $L$-separated types is reflective.
The material in this section is not needed in the rest of the paper,
since we later specialize to the case where $L$ is localization with respect
to a family $f$ of maps. In this case $L' = L_{\suspsym f}$,
which is known to be reflective.
Nevertheless, the more general existence we prove in this section
is likely to be of use in other situations.
The construction we give also has the technical advantage that it
uses only non-recursive higher inductive types, rather than the
recursive higher inductive types employed in~\cite{RSS}.

We will need a general lemma that allows us to construct extensions along maps.
This follows from~\cite[Lemma~1.49]{RSS}, but we include a proof sketch since
it is more direct in this situation.

\begin{lem}\label{lem:unique_extension}
Let $A$, $B$ and $C$ be types, and
let $g:A\to B$ and $f:A\to C$ be maps for which we have a unique extension
\[
  \begin{tikzcd}
    \fib{g}{b} \arrow[r,"f\circ\pr_1"] \arrow[d] & C \\
    \unit \arrow[ur,dashed]
  \end{tikzcd}
\]
for every $b:B$.
Then $f$ extends uniquely along $g$.
\end{lem}

\begin{proof}
By assumption we have
\[
  \prd{b:B}\mathsf{isContr}\Big(\sm{c:C}\prd{a:A}{p:g(a)=b} f(a)=c\Big).
\]
The center of contraction gives us an extension
\[
  \begin{tikzcd}
    A \arrow[r,"f"] \arrow[d,swap,"g"] & C \\
    B \arrow[ur,dashed,swap,"\tilde{f}"]
  \end{tikzcd}
\]
and its uniqueness follows from the contraction.
\end{proof}

Next we give a sufficient condition for a map to be an $L'$-localization.

\begin{prop}\label{lemma:sufficientforlocalization}
Let $X$ be a type.
If $\eta': X \to X'$ is a surjective map such that for any $x,y:X$,
\[
  \mathsf{ap}_{\eta'}:(x=y) \lra (\eta'(x)=\eta'(y))
\]
is an $L$-localization, then $\eta'$ is an $L'$-localization.
\end{prop}

\begin{proof}
By assumption, the types $\eta'(x) = \eta'(y)$ are $L$-local for every $x, y : X$.
Since $\eta'$ is surjective, it follows that $x' = y'$ is $L$-local
for every $x', y' : X'$.
That is, $X'$ is $L$-separated.

It remains to show that $\eta'$ is universal, so assume given $f : X \to Y$
with $Y$ $L$-separated.
By \cref{lem:unique_extension}, it is enough to show that $f$ restricts to a unique constant map
on the fibers of $\eta'$. This means that we must show that
\[
  \sm{y:Y}\prd{x:X} (\eta'(x)=x') \lra (f(x)=y)
\]
is contractible for every $x':X'$.
Since this is a mere proposition, and $\eta'$ is surjective, we can assume that
$x' = \eta'(z)$.  So it is enough to show that
\[
  \sm{y:Y}\prd{x:X} (\eta'(x)=\eta'(z)) \lra (f(x)=y)
\]
is contractible for every $z:X$.
Since $Y$ is assumed to be $L$-separated and $\mathsf{ap}_{\eta'}$ is assumed to be an $L$-localization,
this type is equivalent to
\[
  \sm{y:Y}\prd{x:X} (x=z) \lra (f(x)=y)
\]
and it is easy to see that this is a contractible type by applying the contractibility of
the total space of the path fibration twice.
\end{proof}

Now we can prove the main result of this section.

\begin{thm}\label{thm:Lsep}
For any reflective subuniverse $L$, the subuniverse of $L$-separated types is again reflective.
\end{thm}

\begin{proof}
Fix a type $X : \UU$. Let $\mathcal{Y}_L:X\to (X\to\UU)$ be given by
\[
\mathcal{Y}_L(x,y)\defeq L(y=x).
\]
We would like to define $L' X$ to be $\im(\mathcal{Y}_L)$, but this is a subtype of
$X \to \UU$, so it is not small (i.e., it does not live in $\UU$).
However, since $\UU$ is locally small, so is $X \to \UU$.
Thus the join construction~\cite{joinconstruction} implies that the image
is equivalent to a small type which we denote $L' X$.
This comes equipped with a surjective map
\[
\eta':X \lra L'X,
\]
which we take to be the unit of the reflective subuniverse.

To show that $\eta'$ is a localization, we apply \cref{lemma:sufficientforlocalization}.
First we show that $L' X$ is $L$-separated.  Since $\eta'$ is surjective
and being $L$-local is a proposition, it is enough to show that
$\eta'(x) = \eta'(y)$ is $L$-local for $x$ and $y$ in $X$.
Since $L' X$ embeds in $X \to \UU$, we have an equivalence between
$\eta'(x) = \eta'(y)$ and $(\lambda z . L(z = x)) = (\lambda z . L(z = y))$.
The latter is equivalent to $\prd{z:X} L(z = x) = L(z = y)$, which
is $L$-local by \cref{prop:UU_L-is-L-separated}.

It remains to show that the canonical map $\phi : L(x=y)\to (\eta'(x)=\eta'(y))$ is an equivalence.
For this, it suffices to show that $i \circ \phi$ is an equivalence for some embedding $i$.
Indeed, if $i$ is an embedding and $i \circ \phi$ is an equivalence,
then $i$ is surjective.  Therefore $i$ is an equivalence and hence so is $\phi$.
We proceed to construct an embedding $i$ such that $i \circ \phi$ is the identity map.
We define $i$ as the following composite:
\begin{align*}
                 \eta'(x) = \eta'(y)
&\,\simeq\,      (\lambda z . L(z = x)) = (\lambda z . L(z = y)) \\
&\,\simeq\,      \prd{z:X} L(z = x) = L(z = y) \\
&\,\simeq\,      \prd{z:X} L(z = x) \simeq L(z = y) \\
&\hookrightarrow \prd{z:X} L(z = x) \to L(z = y) \\
&\,\simeq\,      \prd{z:X} (z = x) \to L(z = y) \\
&\,\simeq\,      L(x = y) .
\end{align*}
The first equivalence is $\ap{\iota}$, where $\iota$ is the embedding
of $L'(X)$ into $X \to \UU$.
The second equivalence is function extensionality, and the third is univalence.
The embedding is the forgetful map from equivalences to maps.
The next map is composition with $\eta$, which is an equivalence
because $L(x = z)$ is $L$-local.
The last map is evaluation at $\refl{x}$, and is an equivalence
by path induction.

To show that $i \circ \phi$ is the identity, it suffices to show
that $i \circ \phi \circ \eta = \eta$, as maps $(x = y) \to L(x = y)$.
When $x \equiv y$, one can check directly that $i(\phi(\eta(\refl{x}))) =
i(\ap{\eta'}(\refl{x})) = i(\refl{\eta'(x)}) = \eta(\refl{x})$,
and so by path induction we are done.
\end{proof}

\subsection{$L'$-localization and finite limits}\label{ss:lex}

We now explain how $L$ and $L'$ together behave similarly to a \define{left exact (lex) modality},
i.e., a modality that preserves pullbacks.
Theorem~3.1 of \cite{RSS} gives 13 equivalent characterizations of a lex modality,
and it turns out that these hold for any reflective subuniverse
if the modal operator is replaced by $L$ and $L'$ in the appropriate way.
The propositions in this section show this for parts (ix), (x), (xii) and (xi)
of Theorem~3.1, respectively.
The proofs use the dependent elimination of $L'$ in a crucial way,
but do not use the specific construction of $L'$-localization, just the existence.

Note that $L'$ itself is not necessarily lex, even if $L$ is lex.
For example, $(-2)$-truncation is lex, but $(-1)$-truncation is not.

We start with a characterization of the identity types of an $L'$-localization.
This is a generalization of~\cite[Lemma~7.3.12]{hottbook}.

\begin{prop}\label{Lseparatedloopspace} %
    Given a type $X$ and points $x,y : X$, the unique map $\phi : L(x = y) \to (\eta' x = \eta' y)$ making the triangle
    \[
    \begin{tikzpicture}
      \matrix (m) [matrix of math nodes,row sep=2em,column sep=3em,minimum width=2em]
      { (x=y) & \\
        L(x = y)& (\eta' x = \eta' y) \\};
      \path[->]
        (m-1-1) edge node [left] {$\eta$} (m-2-1)
                edge node [above] {$\mathsf{ap}_{\eta'}$} (m-2-2)
        (m-2-1) edge node [above] {$\sim$} node [below] {$\phi$} (m-2-2)
        ;
    \end{tikzpicture}\vspace*{-6pt}
    \]\enlargethispage{5pt}%
    commute is an equivalence.
In particular, when $X$ is pointed,
$\loopspacesym L' X \simeq L \loopspacesym X$.
\end{prop}

\begin{proof}
    It suffices to produce an equivalence $\phi'$ going in the other direction which
    makes the triangle commute.
    In order to define this map, we generalize the target $L(x = y)$.
    Fix $x : X$.
    By \cref{lemma:extendtoUL}, the type family $X \to \UU_L$ sending
    $y$ to $L(x = y)$ extends to $L'X$
    as follows:
\begin{equation*}
\begin{tikzcd}[column sep=huge]
X \arrow[r,"y\mapsto L(x=y)"] \arrow[d,swap,"{\eta'}"] & \UU_L \\
L'X . \arrow[ur,dashed,swap,"P"]
\end{tikzcd}
\end{equation*}
Now we can define $\phi' : (\eta' x = y') \to P(y')$ by path induction, sending
$\refl{\eta' x}$ to $\eta(\refl{x})$.
The triangle involving $\eta$ and $\mathsf{ap}_{\eta'}$ commutes by path induction.
By~\cite[Theorem~5.8.2]{hottbook}, to show that $\phi'$ is an equivalence
(condition (iii) of that theorem), it suffices to prove that
$\sm{y':L'X} P(y')$ is contractible (condition (iv)).%
\footnote{This paragraph has been corrected
and simplified compared to the published version of this paper.
Thanks to Mike Shulman for pointing out the gap.}

For the center of contraction we take $(\eta'(x),\eta(\refl{x}))$.
It remains to construct a contraction
\begin{equation*}
\prd{y':L'X}{p:P(y')}(\eta'(x),\eta(\refl{x}))=(y',p).
\end{equation*}
By \cref{lemma:separatedpluslocalisseparated}, the total space of $P$ is $L$-separated, 
so it follows that 
\[
  \prd{p:P(y')} (\eta'(x),\eta(\refl{x}))=(y',p)
\]
is $L$-local for every $y':L'X$. 
Thus \cref{proposition:inductionLseparated} reduces the problem to
constructing a term of type
\begin{equation*}
\prd{y:X}{p:L(x=y)}(\eta'(x),\eta(\refl{x}))=(\eta'(y),p).
\end{equation*}

On the other hand, for $y : X$ we have equivalences
\begin{align*}
    \left(\sm{p:L(x=y)}(\eta'(x),\eta(\refl{x}))=(\eta'(y),p)\right)
 & \eqvsym \left(\sm{p:L(x=y)}{\alpha:\eta'(x)=\eta'(y)} \trans{\alpha}{\eta(\refl{x})} = p \right) \\
 & \eqvsym \sm{\alpha:\eta'(x) = \eta'(y)} \unit \\
 & \eqvsym \left( \eta'(x) = \eta'(y) \right),
\end{align*}
where the last type is clearly $L$-local.
So we can apply \cref{theorem:generalized-induction} to reduce the problem to
constructing a term of type
\begin{equation*}
    \prd{y:X}{p:x=y}(\eta'(x),\eta(\refl{x}))=(\eta'(y),\eta(p)) ,
\end{equation*}
which we can do using path induction.
\end{proof}

Before proving the next result,
we need a lemma, which follows directly from the dependent elimination of $L'$.

\begin{lem}\label{lemma:Lequivalencetotalspaces}
Let $P:L'X\to \UU$ be a type family over $L'X$. 
Then the map
\begin{equation*}
f:\Big(\sm{x:X} P(\eta'(x))\Big)\to \Big(\sm{y:L'X}P(y)\Big)
\end{equation*}
given by $(x,p)\mapsto (\eta'(x),p)$ is an $L$-equivalence. 
\end{lem}

\begin{proof}
Assume given an $L$-local type $Z$ and notice that we have the following factorization of $\precomp{f}$:
    \begin{align*}
        \left(\sm{y : L'X} P(y)\right) \to Z &\simeq \prd{y : L'X} P(y) \to Z\\
                                         &\simeq \prd{x : X} P(\eta' x) \to Z \\
                                         &\simeq \left(\sm{x:X} P(\eta' x)\right) \to Z .
    \end{align*}
In the second equivalence, we use \cref{proposition:inductionLseparated} together with the fact
that $Z$ and $P(y) \to Z$ are $L$-local.
\end{proof}

\begin{prop}\label{remark:preservationpullbacks} %
    Given a cospan $Y \xrightarrow{g} Z \xleftarrow{f} X$, let $P$ denote its pullback.
    Let $Q$ denote the pullback of the $L'$-localized cospan $L'Y \rightarrow L'Z \leftarrow L'X$.
    Then the map $P \to Q$ induced by the naturality of $L'$-localization is an $L$-equivalence.
\end{prop}
\begin{proof}
    The result follows by observing that the $L$-localization of the map $P \to Q$
    can be factored as the following chain of equivalences:
    \begin{align*}
         LP &\simeq L \left( \sm{x:X}\sm{y:Y} \, f(x) = g(y)\right)\\
            &\simeq L \left( \sm{x:X}\sm{y:Y} \, L (f(x) = g(y))\right)\\
            &\simeq L \left( \sm{x:X}\sm{y:Y} \, \eta'(f(x)) = \eta'(g(y))\right)\\
            &\simeq L \left( \sm{x:X}\sm{y:Y} \, L'f(\eta'(x)) = L'g(\eta'(y))\right)\\
            &\simeq L \left( \sm{x':L'X}\sm{y':L'Y} \, L'f(x') = L'g(y')\right) \simeq LQ .
    \end{align*}
    Here we used~\cite[Theorem~1.24]{RSS} in the second equivalence and \cref{Lseparatedloopspace} in the third one.
    For the fourth equivalence we use the naturality squares of $L'$, while the fifth equivalence
    follows from \cref{lemma:Lequivalencetotalspaces}.
\end{proof}

As a consequence we get a result about the preservation of certain fiber sequences.

\begin{cor}\label{corollary:preservationfibersequences2}
    Given a fiber sequence $F \to Y \xrightarrow{f} X$, there is a map of fiber sequences
    \[
        \begin{tikzpicture}
          \matrix (m) [matrix of math nodes,row sep=2em,column sep=3em,minimum width=2em]
            { F & Y & X \\
              F' & L'Y & L'X \\};
          \path[->]
            (m-1-1) edge [right hook->] node [right] {} (m-1-2)
                    edge node [right] {} (m-2-1)
            (m-1-2) edge node [above] {$f$} (m-1-3)
                    edge node [right] {$\eta'$} (m-2-2)
            (m-2-2) edge node [above] {$L'f$} (m-2-3)
            (m-2-1) edge [right hook->] node [right] {} (m-2-2)
            (m-1-3) edge node [right] {$\eta'$} (m-2-3)
            ;
        \end{tikzpicture}
    \]
    in which the left vertical map is an $L$-equivalence.
    Here $F$ is the fiber over some $x_0 : X$, and $F'$ is the fiber over $\eta'(x_0)$.\qed
\end{cor}

\begin{prop}\label{cor:L'equivalenceisLconnected} %
    Every $L'$-equivalence $f : Y \to X$ is $L$-connected.
    In particular, $\eta' : X \to L'X$ is $L$-connected.
\end{prop}

\begin{proof}
    Let $x : X$.  We must show that the $L$-localization of $F$, the fiber of $f$ over $x : X$, is contractible.
    To prove this we use \cref{corollary:preservationfibersequences2},
    which gives us an $L$-equivalence between $F$ and $F'$, the fiber of $L'f$ at $\eta' x$.
    The map $L' f$ is an equivalence by hypothesis, so $F' \simeq 1$, and thus $F$ is $L$-connected.
\end{proof}

While the converse of \cref{cor:L'equivalenceisLconnected} does not hold,
it is shown in~\cite[Lemma~1.35]{RSS} that every $L$-connected map is an $L$-equivalence.
The reader may find the following diagram of implications helpful:
\[
\begin{tikzcd} %
  L'\text{-connected} \arrow[r,Rightarrow] \arrow[d,Rightarrow]
           & L'\text{-equivalence} \arrow[d,Rightarrow] \arrow[dl,Rightarrow] \\
  L\text{-connected} \arrow[r,Rightarrow]     & L\text{-equivalence} .
\end{tikzcd}
\]
The vertical implications follow from the fact that every $L$-local type is $L'$-local.

\begin{prop}\label{proposition:2outof3} %
    Given maps $f : Y \to X$ and $g : X \to Z$ such that $g \circ f$ is $L'$-connected,
    then $f$ is $L$-connected if and only if $g$ is $L'$-connected.
\end{prop}

\begin{proof}
    If $g$ is $L'$-connected, then both $g\circ f$ and $g$ are $L'$-equivalences,
    and thus $f$ is an $L'$-equivalence. Then \cref{cor:L'equivalenceisLconnected}
    implies that $f$ is $L$-connected.

    For the converse, notice that taking fibers over each $z:Z$ reduces the problem
    to showing that given an $L$-connected map $f : Y \to X$
    such that $Y$ is $L'$-connected, it follows that $X$ is $L'$-connected.

    Since $L'Y$ is contractible, it is enough to show that $L'f$ is an equivalence.
    To do this, we prove that the fibers of $L'f$ are contractible.
    Since this a proposition, and $\eta' : X \to L'X$ is surjective (\cref{lemma:etasurjective}), it is enough
    to show that, for each $x : X$, the fiber $F'$ of $L'f$ at $\eta'(x)$ is contractible.
    Notice that $F'$ is $L$-local, since $L' Y$ being contractible implies that $F'$ is
    equivalent to an identity type of $L' X$.
    Now, using \cref{corollary:preservationfibersequences2},
    we get an $L$-equivalence between the fiber of $f$ at $x$, and $F'$.
    But since $F'$ is also $L$-local, this is in fact an $L$-localization.
    Using the fact that $f$ is $L$-connected, we deduce that $F'$ is contractible.
\end{proof}

\begin{rmk}
The above proposition almost gives us a $2$-out-of-$3$ property that combines $L$ and $L'$.
However the map $\emptyt \to \unit$ is $(-2)$-connected, and $\unit$ is $(-1)$-connected,
whereas $\emptyt$ is not $(-1)$-connected. So the remaining implication of the $2$-out-of-$3$ property
does not hold. One can show the weaker result that the composite of an $L$-connected map followed by
an $L'$-connected map is $L$-connected.
\end{rmk}

\begin{rmk}
\cref{Lseparatedloopspace} allows us to give a concrete description of
the extension defined in \cref{lemma:extendtoUL}.
Using an argument similar to the one used in the proof of \cref{remark:preservationpullbacks},
one can show that given an $L'$-localization $\eta' : X \to L' X$
and a map $f: Y \to X$ with $L$-local fibers, $f$ is the pullback of
the fiberwise $L$-localization of $\eta' \circ f$.
\end{rmk}

We conclude this section with a characterization of $L'$-localizations.

\begin{thm}\label{theorem:commutativityloopreflectivesubuniv}
Consider a reflective subuniverse $L$, and let $X$ be a type.
For a map $\eta':X\to X'$, the following are equivalent:
\begin{enumerate}
\item $\eta':X\to X'$ is an $L'$-localization.
\item The map $\eta':X\to X'$ is surjective and for any $x,y:X$,
\[
  \mathsf{ap}_{\eta'}:(x=y) \lra (\eta'(x)=\eta'(y))
\]
is an $L$-localization.
\end{enumerate}
\end{thm}

\begin{proof}
Assume that (1) holds.
The map $\eta'$ is surjective by \cref{lemma:etasurjective}.
The other claim follows from \cref{Lseparatedloopspace}.

The other implication is \cref{lemma:sufficientforlocalization}.
\end{proof}

\begin{rmk}
It follows that a type $X$ is $L'$-connected if and only if
it is merely inhabited and $x_1 = x_2$ is $L$-connected for every $x_1, x_2 : X$.
More generally, using that the identity types of the fibers of $f : A \to B$ are
equivalent to fibers of $A \to A \times_B A$, 
we see that a map $f : A \to B$ is $L'$-connected if and only if
$f$ is surjective and the diagonal $A \to A \times_B A$ is $L$-connected.
\end{rmk}

\section{Localization with respect to a family of maps}\label{section:localization}

In this section, we discuss localization with respect to a family of maps.
Our primary examples are localization at the degree $k$ map from $\sphere{1}$ to $\sphere{1}$
and localization at a family of such maps.

In \cref{ss:basic-properties}, we prove some basic properties and study some consequences
of the general theory of separated subuniverses in the case of localization with respect to families of maps.
In \cref{ss:connectedness}, we give conditions under which
$f$-localization preserves coproducts and connectedness.
Finally, \cref{ss:orthogonal-factorization-systems} contains results about the interaction between
orthogonal factorization systems and localizations at families of maps, generalizing some
previous results. The results in \cref{ss:orthogonal-factorization-systems} are not used in the rest
of the paper.

\subsection{Local types and their properties}\label{ss:basic-properties}

We recall the following definitions from~\cite{RSS}.
\begin{defn}\ 
\begin{enumerate}
\item    Let $f: \prd{i:I} A_i \to B_i$ be a family of maps indexed by a type $I$.
    A type $X$ is \define{$f$-local} if
    $\precomp{f_i} : (B_i \to X) \to (A_i \to X)$ is an equivalence for every $i : I$.
\item Let $A:I\to\UU$ be a family of types. A type $X$ is said to be \define{$A$-null} if it is $u$-local for the family of maps $u:\prd{i:I}A_i\to\unit$. 
\end{enumerate}
\end{defn} 

As previously noted, \cite[Theorem 2.18]{RSS} shows that for every
family $f$, the $f$-local types form a reflective subuniverse.
The localization $L_f X$ is constructed as a higher inductive type,
and we write $\eta_X : X \to L_f X$ for the unit.
In the case of a family $A:I\to\UU$, localization at the unique family $u:\prd{i:I}A_i \to \unit$ is called \define{$A$-nullification}.
By \cite[Theorem 2.19]{RSS}, the reflective subuniverse of $A$-null types is stable under dependent sums and therefore $A$-nullification is a modality (\cref{modality}).

\begin{eg} We recall the following basic examples from~\cite{RSS}.
\begin{enumerate}
\item The unit type is $f$-local for any map $f$.
\item A type $X$ is $(\emptyt\to\unit)$-local if and only if $X$ is contractible.
\item A type $X$ is $(\sphere{n+1}\to\unit)$-local if and only if $X$ is $n$-truncated.
\end{enumerate}
\end{eg}

When $f$ consists of pointed maps between pointed types, we can test whether
a type $X$ is local using the pointed mapping types.

\begin{lem}\label{lemma:pointed}
    If $f : \prd{i:I} A_i \to B_i$ is a family of pointed maps between pointed types,
    then a type $X$ is $f$-local if and only if
    $\precomp{f_i} : (B_i \pto X) \to (A_i \pto X)$ is an equivalence
    for every base point $x : X$ and every $i : I$.
    If $X$ is connected, then it is enough to check this for one $x : X$.
\end{lem}

\begin{proof}
    The second claim follows from the first one, since
    $\mathsf{isEquiv}$ is a mere proposition.
    For the first claim, fix $i : I$ and consider the diagram
    \[
        \begin{tikzpicture}
          \matrix (m) [matrix of math nodes,row sep=2em,column sep=3em,minimum width=2em]
          { (B_i \pto X) & (B_i \to X) & X \\
            (A_i \pto X) & (A_i \to X) & X , \\};
          \path[->]
            (m-1-1) edge [right hook->] node [above] {} (m-1-2)
                    edge node [left] {} (m-2-1)
            (m-1-2) edge node [above] {} (m-1-3)
                    edge node [left] {} (m-2-2)
            (m-2-1) edge [right hook->] node [above] {} (m-2-2)
            (m-2-2) edge node [above] {} (m-2-3)
            (m-1-3) edge [double equal sign distance,-] node [left] { } (m-2-3)
            ;
        \end{tikzpicture}
    \]
    where the horizontal sequences are fiber sequences with the fiber taken over
    some $x : X$. Since fiberwise maps are equivalences exactly when they are fiberwise
    equivalences, the vertical map in the middle is an equivalence if and
    only if the vertical map on the left is an equivalence for every $x : X$.
\end{proof}

Since the pointed mapping space $\unit\to_\ast X$ is contractible for any pointed space $X$, we have the following corollary.

\begin{cor}\label{cor:pointed_null}
If $A$ is a family of pointed types, then a type $X$ is $A$-null
if and only if $(A_i \pto X)$ is contractible
for every base point $x : X$ and every $i : I$. \qed
\end{cor}

We turn to a comparison of $L_f$ and $L_{\suspsym f}$.
The next corollary follows immediately from \cref{lemma:characterizationsigmaflocal}
and \cref{Lseparatedloopspace}.

\begin{cor}\label{remark:commutativitylooplocalization}
    Let $f : \prd{i:I} A_i \to B_i$ be a family of maps.
    For any pointed type $X$, we have that $\loopspacesym \eta : \loopspacesym X \to \loopspacesym L_{\susp{f}} X$
    is an $f$-localization. In particular, $\loopspacesym L_{\susp{f}} X \simeq L_f \loopspacesym X$. \qed
\end{cor}

This result is a type theoretic analog of \cite[Theorem 3.1]{Bousfield} and \cite[3.A.1]{DrorFarjoun}.
Interestingly, the classical proofs use delooping techniques which are not yet available in type theory.
Our proof instead makes essential use of a univalent universe.

Given this corollary, we are led to consider the relationship between $f$-local types and $\suspsym f$-local types, as a step towards comparing $L_f$ and $L_{\suspsym f}$.

\begin{thm}\label{theorem:characterizinglocalness}
    Let $f: \prd{i:I} A_i \to B_i$ be a family of pointed maps between pointed types and let $n \geq 1$.
    Consider the following conditions on a type $X$:
    \begin{enumerate}
    \item $X$ is $f$-local.
    \item $X$ is $\suspsym^{n-1} C_f$-null,
    where $C_f$ is the family of cofibers of the family $f$. 
    \item $X$ is $\suspsym^{n} f$-local.
    \end{enumerate}
    Then (1) $\implies$ (2) $\implies$ (3).
    Moreover, if the pointed mapping spaces $(A_i \pto X)$ and $(B_i \pto X)$
    are $(n-1)$-connected for all choices of base point $x : X$
    and every $i : I$,
    then the three conditions are equivalent.
\end{thm}

\begin{proof}
We will show the required implications for each $i : I$.
So, without loss of generality we consider a single map $f : A \to B$.

By \cref{lemma:pointed}, it suffices to consider pointed mapping
spaces for each $x : X$.
With $x$ chosen, we have a long fiber sequence
\[
  \begin{tikzpicture}[descr/.style={fill=white,inner sep=1.5pt}]
        \matrix (m) [matrix of math nodes,row sep=2em,column sep=3em,minimum width=2em]
        {   &  (A\pto X) & {(B\pto X)} & (C_f \pto X) \\
            & (\suspsym A \pto X) & (\suspsym B \pto X)& (\suspsym C_f \pto X) \\
            & \hspace*{2em}\cdots\hspace*{2em} \\
        };

        \path[overlay,->, font=\scriptsize]
        (m-1-3) edge (m-1-2)
        (m-1-4) edge (m-1-3)
        (m-2-2) edge[out=175,in=355] node[descr,yshift=0.3ex]{}(m-1-4)
        (m-2-3) edge (m-2-2)
        (m-2-4) edge (m-2-3)
        (m-3-2) edge[out=175,in=355] node[descr,yshift=0.3ex] {} (m-2-4);
  \end{tikzpicture}
\]

Assuming (1), from the first fiber sequence in the diagram above
it follows that $(C_f \pto X)$ is contractible. This means that $X$ is $C_f$-null and
thus $\suspsym^{n-1} C_f$-null, since any $C_f$-null type is also $C_f$-separated. Thus (1) implies (2).

To see that (2) implies (3), consider a piece of the above fiber sequence:
\[
  (\suspsym^{n} B \pto X) \xrightarrow{\precomp{(\suspsym^n f)}}
  (\suspsym^n A\pto X) \lra (\suspsym^{n-1} C_f \pto X).
\]
If (2) holds, then the base of the fiber sequence is contractible and thus the inclusion of the
fiber in the total space
is an equivalence, proving (3).

Finally, we show that (3) implies (1) under the connectedness hypotheses.
Notice that we can express
$\precomp{(\suspsym^{n} f)} : (\suspsym^n B \pto X) \to (\suspsym^n A \pto X)$ as
\[
    (\suspsym^n B\pto X) \simeq \loopspacesym^n(B \pto X) \xrightarrow{\loopspacesym^n (\precomp{f})}
    \loopspacesym^n(A \pto X) \simeq (\suspsym^n A \pto X),
\]
using~\cite[Lemma~6.5.4]{hottbook}.
If $X$ is $\suspsym^{n}f$-local, it follows that $\loopspacesym^n (\precomp{f})$ is an equivalence.
By hypothesis, $(A\pto X)$ and $(B\pto X)$ are $(n-1)$-connected, and so by~\cite[Corollary~8.8.2]{hottbook} applied $n$ times
it follows that $\precomp{f} : (B \pto X) \to (A \pto X)$ is an equivalence.
\end{proof}

\subsection{Preservation of coproducts and connectedness}\label{ss:connectedness}

We begin this section by studying conditions under which
$f$-localization preserves coproducts. 
By a coproduct, we mean a set-indexed $\sum$-type.
We first prove a lemma, which will also be used in \cref{ss:localizationofgroups}.

\begin{lem}\label{lemma:setsarelocal}
    Let $f : \prd{i:I} A_i \to B_i$ be a family of maps between connected types.
    Then a coproduct of $f$-local types is $f$-local.
    In particular, sets are $f$-local.
\end{lem}

\begin{proof}
    It suffices to prove the lemma for each $f_i$, so we assume
    that we are given a single map $f : A \to B$.
    Let $J$ be a set and let $T : J \to \UU$ be a type family.
    Since $A$ is connected, we have $(A \to J) \simeq (\trunc{0}{A} \to J) \simeq (1 \to J) \simeq J$,
    so we deduce that $(A \to \sum_j T_j) \simeq \sm{\alpha : A \to J} \prd{a : A} T_{\alpha(a)} \simeq \sum_j (A \to T_j)$.
    The same is true for $B$, so
    we can factor the map $\precomp{f} : (B \to \sum_j T_j) \to (A \to \sum_j T_j)$ as
    \[  \textstyle
        \left(B \to \sum_j T_j \right) \simeq \sum_j (B \to T_j) \simeq \sum_j (A \to T_j) \simeq \left(A \to \sum_j T_j\right),
    \]
    which shows that $\sum_j T_j$ is $f$-local.

    Since $\unit$ is $f$-local, a special case is that $\sum_j \unit \simeq J$ is $f$-local.
\end{proof}

\begin{cor}
    Let $f : \prd{i:I} A_i \to B_i$ be a family of maps between connected types.
    Then $f$-localization preserves coproducts.
\end{cor}

\begin{proof}
    Let $J$ be a set and let $S : J \to \UU$ be a type family.
    Consider the coproduct $\sum_j S_j$.
    By \cref{lemma:setsarelocal}, $\sum_j L_f S_j$ is $f$-local. We claim that the natural
    map $\sum_j S_j \to \sum_j L_f S_j$ is a localization.
    Let $Y$ be $f$-local. Then we have
    \[  \textstyle
        \Bigg( \bigg( \sum_j S_j \bigg) \to Y \Bigg) \simeq \prod_j (S_j \to Y) \simeq
        \prod_j (L_f S_j \to Y) \simeq \Bigg( \bigg( \sum_j L_f S_j \bigg) \to Y \Bigg).\qedhere
    \]
\end{proof}

\medskip

Our next goal is to give conditions under which $f$-localization
preserves $n$-connected\-ness.

\begin{prop}\label{prop:Kequiv-Lf}
Let $K$ be a reflective subuniverse and let $f$ be a family of $K$-equiva\-lences.
Then $\eta : X \to L_f X$ is a $K$-equivalence for every $X$.
\end{prop}

\begin{proof}
Since each $f_i$ is a $K$-equivalence, every $K$-local type is $f$-local,
by \cref{lemma:characterizationorthogonal}.
Therefore, $\eta : X \to L_f X$ is a $K$-equivalence
by \cref{lemma:comparelocalization}(1).
\end{proof}

\begin{cor}\label{cor:preserve-n-connected}
For $n \geq -1$, let $f$ be a family of maps such that each $\trunc{n}{f_i}$ is an equivalence.
If $X$ is $n$-connected, then $L_f X$ is $n$-connected.
\end{cor}

\begin{proof}
Let $K$ be the subuniverse of $n$-truncated types.
Then each $f_i$ is a $K$-equivalence, by assumption.
It follows from \cref{prop:Kequiv-Lf} that $X \to L_f X$ is a $K$-equivalence.
Thus, if $\trunc{n}{X}$ is contractible, so is $\trunc{n}{L_f X}$.
\end{proof}

\subsection{Orthogonal factorization systems}\label{ss:orthogonal-factorization-systems}

In this section, we strengthen \cref{cor:preserve-n-connected}, using
the framework of orthogonal factorization systems.
While we do not need this generalization in the rest of the paper,
the stronger results will likely be of independent interest.

Roughly speaking, an \define{orthogonal factorization system} consists of
classes $\LLL$ and $\RRR$ of maps such that every map factors uniquely as
$r \circ l$, with $l$ in $\LLL$ and $r$ in $\RRR$.
See~\cite{RSS} for a detailed account of orthogonal factorization systems in type theory.
The reader not familiar with orthogonal factorization systems can assume
that $\LLL$ is the class of $n$-connected maps
and $\RRR$ is the class of $n$-truncated maps for some $n \geq -2$;
this case is treated in~\cite[Section 7.6]{hottbook}.
(For example, when $n = -1$, $\LLL$ consists of the surjective maps
and $\RRR$ consists of the embeddings.)

\begin{lem}\label{lemma:subtypes}
    Let $(\LLL, \RRR)$ be an orthogonal factorization system and
    let $f : \prd{i:I} A_i \to B_i$ be a family of maps in $\LLL$.
    If $r : S \to X$ is in $\RRR$ and $X$ is $f$-local, then $S$ is $f$-local.
    In particular, if each $f_i$ is surjective,
    then any subtype of an $f$-local type is $f$-local.
\end{lem}

\begin{proof}
    Suppose that $X$ is $f$-local
    and that $r : S \to X$ is in $\RRR$.
    Fix $i : I$.
    Since $f_i$ is in $\LLL$ and $r$ is in $\RRR$, the following square is a pullback:
    \[
        \begin{tikzpicture}
          \matrix (m) [matrix of math nodes,row sep=2em,column sep=3em,minimum width=2em]
          { (B_i \sto S) & (A_i \sto S) \\
            (B_i \sto X) & (A_i \sto X) . \\};
          \path[->]
            (m-1-1) edge node [above] {$\precomp{f_i}$} (m-1-2)
                    edge node [left] {$r \circ \blank$} (m-2-1)
            (m-1-2) edge node [right] {$r \circ \blank$} (m-2-2)
            (m-2-1) edge node [below] {$\precomp{f_i}$} (m-2-2)
            ;
        \end{tikzpicture}
    \]
    This follows from~\cite[Lemma~1.44]{RSS}, which says that the fibers of
    the map comparing $B_i \sto S$ to the pullback are contractible.
    The bottom arrow is an equivalence since $X$ is $f$-local,
    so the top arrow is an equivalence, as required.
\end{proof}

\begin{thm}\label{theorem:localizationpreservesconnected}
    Let $(\LLL, \RRR)$ be an orthogonal factorization system and
    let $f$ be a family of maps in $\LLL$.
    Then $\eta : X \to L_f X$ is in $\LLL$ for every $X$.
\end{thm}

\begin{proof}
    Factor the unit map as $X \xrightarrow{l} I \xrightarrow{r} L_f X$, with $l$ in $\LLL$
    and $r$ in $\RRR$.
    By \cref{lemma:subtypes},
    $I$ is $f$-local. This gives us a map $\overline{l} : L_f X \to I$
    such that $r \circ \overline{l} \circ \eta = r \circ l = \eta$.
    By the universal property of $L_f X$, it must be the case that $r \circ \overline{l}$
    is homotopic to the identity of $L_f X$.
    Similarly, $\overline{l} \circ r$ is an automorphism of $I$ that respects the
    factorization, so it must be homotopic to the identity of $I$.
    Thus $r$ is an equivalence and so $\eta = r \circ l$ is in $\LLL$.
\end{proof}

\begin{cor}\label{cor:preserve-n-connected-stronger}
    For $n \geq -1$, let $f$ be a family of $(n-1)$-connected maps.
    If $X$ is $n$-connected, then $L_f X$ is $n$-connected.
\end{cor}

\begin{proof}
    Applying \cref{theorem:localizationpreservesconnected} with
    $\LLL$ the $(n-1)$-connected maps and $\RRR$ the $(n-1)$-truncated maps
    gives that $\eta : X \to L_f X$ is $(n-1)$-connected.
    Using the forward implication of \cref{proposition:2outof3} with $L$ being $(n-1)$-truncation,
    we see that when $X$ is $n$-connected, so is $L_f X$.
\end{proof}

\cref{prop:Kequiv-Lf} and \cref{theorem:localizationpreservesconnected} are very similar.
To compare them in a concrete situation, take $K$ to be the subuniverse of $n$-truncated
types and consider the orthogonal factorization system in which $\LLL$ is the
class of $n$-connected maps.
If a map $g$ is $n$-connected, then $Kg = \trunc{n}{g}$ is an equivalence, so
\cref{theorem:localizationpreservesconnected} makes a stronger assumption on $f$
and gives a stronger conclusion about $\eta : X \to L_f X$.
Thus neither result implies the other.

On the other hand, \cref{cor:preserve-n-connected-stronger} has the same
conclusion as \cref{cor:preserve-n-connected}, but makes a weaker hypothesis,
since if $\trunc{n}{f}$ is an equivalence then $f$ is $(n-1)$-connected.
This is strictly weaker.  For example, if $n = 0$, the former is the
condition that $f$ is a bijection on components, while the latter is
the condition that $f$ is surjective on components.

\section{Localization away from sets of numbers}\label{section:localizationaway}

In this section, we focus on localization with respect to the \define{degree $k$
map} $\degg(k) : \sphere{1} \to \sphere{1}$, for $k : \N$, and with respect to
families of such maps.
The degree $k$ map is defined by circle induction by mapping $\base$ to $\base$
and $\lloop$ to $\lloop^k$.
Using suggestive language that mirrors the algebraic case, we call $L_{\degg(k)}$
``localization away from $k$'' and say that we are ``inverting $k$.''
We note that, for many applications, one considers localization away from $k$ for $k$ prime.
However, for the results of this section, $k$ can be any natural number.

As might be expected, $\degg(k)$-localizations can be combined to localize a type \emph{at}
a prime $p$, by localizing with respect to the family of maps $\degg(q)$ indexed by all
primes $q$ different from $p$.
With this case in mind, we study localization with respect to any family of degree $k$ maps
indexed by a map $S : A \to \N$ for some type $A$.
In other words, we localize with respect to the family $\degg\circ S : A \to \sphere{1} \to \sphere{1}$.
We denote the family $\degg\circ S$ by $\degg(S)$.

The main goal of this section is to show that, for simply connected types,
$\degg(S)$-localization localizes the homotopy groups away from $S$ (as long as $S$ indexed by $\N$),
as is true in the classical setting.

We begin with a discussion of localization of groups in \cref{ss:localizationofgroups}, which we need in order to describe the effect of localization on homotopy groups.

In \cref{ss:n-conn}, we give characterizations of $\degg(S)$-localization for highly connected types. In particular, we show that the localization of a simply connected type can be computed as the nullification with respect to $M_S$, the family of cofibers of the maps $\degg(S(a))$, or as the localization with respect to $\susp \degg(S)$, the family of suspensions of these maps.
We prove that the $\degg(S)$-localization of a pointed, $(n-1)$-connected type inverts $S$ in $\pi_n$,
and deduce from this that every group has a localization away from $S$.

These characterizations of $\degg(S)$-localization are of interest to us for two reasons. First, the observation that $\degg(S)$-localization can be computed via a nullification implies that, restricted to simply connected types, $\degg(S)$-localization is a modality and therefore is better behaved than an arbitrary localization. Second, the fact that $\degg(S)$-localization can be computed as $\susp \degg(S)$-localization allows us to deduce that, for simply connected types, $\degg(S)$-localization commutes with taking loop spaces, a fact we use in the next section.

In \cref{ss:localizingKgn}, we give a method for computing the $\degg(S)$-localization of a loop space via a mapping telescope construction. This allows us to show that the localization of an Eilenberg-Mac Lane space $K(G,n)$, for $n \geq 1$ and $G$ abelian, is $K(G',n)$, where $G'$ is the algebraic localization of $G$ away from $S$.

In \cref{ss:localization-of-homotopy-groups}, we combine the results of \cref{ss:localizingKgn} and observations about the interaction between $\degg(S)$-localization and truncation to show that localizing a simply connected type at $\degg(S)$ localizes all of the homotopy groups away from $S$.
The results of the last two sections assume that the family $S$ is indexed by the natural numbers.

Finally, in \cref{ss:abelian} we give a direct proof of the fact that
the localization of an abelian group in the category of groups coincides with
its localization in the category of abelian groups, which also follows from
\cref{theorem:localizationKgn}.
This section is not needed in the rest of the paper.

\subsection{Localization of groups}\label{ss:localizationofgroups}

We fix a family $S : A \to \N$ of natural numbers and consider the family
of maps $\degg(S) : A \to \sphere{1} \to \sphere{1}$ sending $a$ to $\degg(S(a))$.

We begin by defining the collection of groups that will be the
``local'' groups for our algebraic localization.
We make the standard assumption that the underlying type of a group is a set.

\begin{defn}
    For $k : \N$, a group $G$ is \define{uniquely $k$-divisible} if the $k$-th power map
    $g \mapsto g^k$ is a bijection.
    A group $G$ is \define{uniquely $S$-divisible} if it is uniquely $S(a)$-divisible for every $a : A$.
\end{defn}

When the group is additive, the $k$-th power map is usually called the ``multiplication by $k$ map.''
Notice that for non-abelian groups the map is not a group homomorphism in general.

Our first goal is to show that the homotopy groups of a $\degg(S)$-local type
are uniquely $S$-divisible.
In order to do this, we begin by characterizing the $\degg(S)$-local types.
For a type $X$, a point $x : X$ and a natural number $k$, we write
$k : \loopspace{X,x} \to \loopspace{X,x}$ for the map sending $\omega$ to $\omega^{S(a)}$,
and we call this map the \define{$k$-fold map}.

\begin{lem}
A type $X$ is $\degg(S)$-local if and only if, for each $x:X$ and $a:A$, the map
$S(a) : \loopspace{X,x} \to \loopspace{X,x}$ is an equivalence.
\end{lem}

\begin{proof}
We apply \cref{lemma:pointed}.
Let $x:X$ and $a:A$.
By the universal property of the circle, we have an equivalence
$(\sphere{1} \pto X) \eqvsym \loopspace{X,x}$,
and it is easy to see that the square
\[
  \begin{tikzcd}
    (\sphere{1} \pto X) \arrow[r,"\simeq"] \arrow[d,swap,"\precomp{\degg(S(a))}"] & \loopspace{X,x} \arrow[d,"S(a)"] \\
    (\sphere{1} \pto X) \arrow[r,swap,"\simeq"] & \loopspace{X,x}
  \end{tikzcd}
\]
commutes.
So we conclude that $\precomp{\degg(S(a))}$ is an equivalence if and only if $S(a)$ is an equivalence.
\end{proof}

\begin{prop}\label{prop:homotopygroupsoflocalarelocal}
    If $X$ is a pointed, $\degg(S)$-local type, then $\pi_{m+1}(X)$ is uniquely $S$-divisible
    for each $m : \N$.
    Conversely, if for some $n \geq 0$, $X$ is a pointed, simply connected $n$-type
    such that $\pi_{m+1}(X)$ is uniquely $S$-divisible for each $m : \N$,
    then $X$ is $\degg(S)$-local.
\end{prop}

For the converse, the requirement that $X$ be truncated is needed because
we need to use Whitehead's theorem.
In fact the converse is also true, and the same proof works, if $X$ is only assumed to be hypercomplete (\cite[Definition~3.18]{RSS}). 

\begin{proof}
    Let $X$ be a pointed, $\degg(S)$-local type and let $a : A$.
\def\fold{\mathsf{pow}} %
    In this proof, we write $\fold(k, Y)$ for the $k$-fold map on $\loopspacesym Y$,
    where $k = S(a)$.
    For every $m : \N$,
    the map $\fold(k, X) : \loopspacesym X \to \loopspacesym X$ induces
    the map $\fold(k, \loopspacesym^m X) : \loopspacesym^m \loopspacesym X \to \loopspacesym^m \loopspacesym X$,
    by the Eckmann-Hilton argument~\cite[Theorem~2.1.6]{hottbook}.
    That is, $\loopspacesym^m \fold(k, X) = \fold(k, \loopspacesym^m X)$.
    It follows that $\loopspacesym^m \fold(k, X)$ induces the $k$-th power map on $\pi_{m+1}( X )$.
    Since $\fold(k, X)$ is an equivalence, the $k$-th power map on $\pi_{m+1}( X )$
    must be an equivalence for each $m$.
    In other words, $\pi_{m+1}( X )$ is uniquely $S$-divisible for each $m$.

    Conversely, suppose that, for some $n \geq 0$, $X$ is a pointed,
    simply connected $n$-type such that $\pi_{m+1}(X)$ is uniquely $S$-divisible
    for each $m : \N$.
    Let $a : A$ and let $k = S(a)$.
    Since $X$ is pointed and connected, it suffices to show that
    $\fold(k, X) : \loopspacesym X \to \loopspacesym X$ is an equivalence.
    By the above argument, $\fold(k, X)$ induces the $k$-th power map on each
    $\pi_{m+1}(X) = \pi_m(\loopspacesym X)$.
    By assumption, these $k$-th power maps are bijections.
    Since $\loopspacesym X$ is a pointed, connected $(n-1)$-type, it follows from the
    truncated Whitehead theorem~\cite[Theorem~8.8.3]{hottbook} that 
    $\fold(k, X)$ is an equivalence.
\end{proof}

Our next goal is to show that a group $G$ is uniquely $S$-divisible if and only
if the Eilenberg--Mac Lane space $K(G, 1)$ is $\degg(S)$-local.
For this, we use the following theorem.

\begin{thm}[{\cite[Theorem 5.1]{BuchholtzDoornRijke}}]\label{theorem:catofgroups}
    Given $n>1$, the pointed, $(n-1)$-connected, $n$-truncated types together with pointed maps
    form a univalent category, and this category is equivalent to the category of abelian groups
    and group homomorphisms.
    Similarly, when $n=1$ we have an equivalence between the category of pointed, $0$-connected,
    $1$-truncated types and the category of groups. \qed
\end{thm}

The equivalence in \cref{theorem:catofgroups} is given by mapping a pointed,
$(n-1)$-connected, $n$-truncated type $X$ to $\pi_n( X )$,
and the inverse is given by mapping a group $G$ to the Eilenberg--Mac Lane
space $K(G,n)$, as constructed in~\cite{FinsterLicata}.
Moreover,~\cite{BuchholtzDoornRijke} shows that the inverse equivalence maps
short exact sequences of groups to fiber sequences of types.

\begin{cor}\label{corollary:characterizationpdivisible}
    For any $n > 1$, an abelian group $G$ is uniquely $S$-divisible if and only if its classifying space $K(G,n)$
    is $\degg(S)$-local. An arbitrary group $G$ is uniquely $S$-divisible if and only if $K(G,1)$ is $\degg(S)$-local.
\end{cor}

\begin{proof}
    As argued in the proof of \cref{prop:homotopygroupsoflocalarelocal}, $\degg(k)$ induces the $k$-th power map under the equivalence of categories, so one morphism is an equivalence if and only if the other is.
\end{proof}

Note that the $n > 1$ case of \cref{corollary:characterizationpdivisible} also follows
immediately from \cref{prop:homotopygroupsoflocalarelocal}.

Now we define what it means to localize a group away from $S$.
This concept is needed to state our main result, \cref{theorem:localizationlocalizes}.

\begin{defn}\label{def:localizationofgroups}
    Given a group $G$, a homomorphism $\eta : G \to G'$ is a \define{localization of $G$ away from $S$
    in the category of groups} if $G'$ is uniquely $S$-divisible and for every group $H$
    which is uniquely $S$-divisible, the precomposition map $\precomp{\eta} : \Hom(G', H) \to \Hom(G, H)$ is an equivalence.
    If $G$ and $G'$ are abelian and $\eta : G \to G'$ has the universal property with respect to uniquely
    $S$-divisible abelian groups, we say that $\eta$ is a \define{localization of $G$ away from $S$ in the category
    of abelian groups}.
\end{defn}

\begin{rmk}\label{rmk:localizationofgroups}
It will follow from \cref{lemma:localizationlocalizesfirst} that every group $G$ has a localization $G \to G'$
away from $S$, obtained by applying $\pi_1$ to the $\deg(S)$-localization of $K(G, 1)$.
Similarly, if $G$ is abelian, its localization in the category of abelian groups is obtained
by applying $\pi_2$ to the $\deg(S)$-localization of $K(G,2)$.
Moreover, by \cref{theorem:localizationKgn}, these algebraic localizations agree
when restricted to abelian groups.
In \cref{ss:abelian}, we give an independent, purely algebraic construction of
the localization of an abelian group away from $S$ which also shows that the
two localizations agree.
\end{rmk}

We conclude this section with three examples.
The first two show that $\degg(k)$-localization is not a modality.
The first one uses a homotopical construction, while the second one is purely algebraic.
The third example shows that the $\degg(k)$-local types are not the separated types for any reflective subuniverse $L$.

\begin{eg}\label{example:nonlocalfib2}
    Consider the map $K(\Z,1) \to K(\Q,1)$ induced by the inclusion of groups $\Z \to \Q$.
    Using the long exact sequence of homotopy groups, one sees that the fiber over any point is
    (merely) equivalent to the set $\Q/\Z$.
    By \cref{lemma:setsarelocal}, sets are $\degg(k)$-local for any $k : \N$,
    so as long as $k > 0$, the type $K(\Z,1)$ is the total space of a fibration with $\degg(k)$-local fibers and $\degg(k)$-local base.
    On the other hand, if $k > 1$, $K(\Z,1)$ is not $\degg(k)$-local, so we see that $\degg(k)$-localization
    is not a modality in general.

In particular, by localizing the fiber sequence $\Q/\Z \to K(\Z,1) \to K(\Q,1)$ we see that
$\degg(k)$-localization does not preserve fiber sequences in general.
\end{eg}

\begin{eg}\label{example:nonlocalfib}
    Let $B$ be the subtype of $\Q \to \Q$ consisting of functions with bounded support.
    We can define this type constructively as the type of functions together with
    a mere bound for their support:
    \[
         B \defeq \sm{f : \Q \to \Q} \, \ttrunc{-1}{\sm{b : \Q} \prd{x : \Q} (x > b) \to (f(x) = 0) \times (f(-x) = 0)}.
    \]
    The group operation is given by $(f + g)(x) = f(x) + g(x)$.
    Notice that both $\Q$ and $B$ are uniquely $k$-divisible for any $k > 0$.

    Consider the semidirect product $P \defeq B \rtimes \Q$, where $\Q$ acts by translation:
    \[
        (r \cdot f)(x) = f(x+r).
    \]
    It can be shown that $(\delta_0, 2) : P$ does not have a square root. Here $\delta_0$ is
    the map $\Q \to \Q$ that is constantly $0$, except on $0$, where it takes the value $1$,
    together with a proof that the support of this map is bounded.
    It follows that if we apply the functor $K(\blank, 1)$ to this short exact sequence of groups,
    we obtain a fiber sequence with connected, $\degg(2)$-local base and fiber,
    but for which the total space is not $\degg(2)$-local.

    Similarly, it can be shown that for any $k > 1$, $P$ is not uniquely $k$-divisible
    and therefore that $\degg(k)$-localization is also not a modality.
\end{eg}

\begin{eg}
As a final example in a similar spirit, we show that for $k > 1$, $L_{\degg(k)}$ is not
of the form $L'$ for any reflective subuniverse $L$.
For any $L$, the $L$-separated types are determined by their identity types.
However, the identity types of $K(\Q,1)$ and $K(\Z,1)$ are equivalent as types,
and the former is $L_{\degg(k)}$-local while the latter is not.
\end{eg}

\subsection{Localizing highly connected types}\label{ss:n-conn}

We continue to fix a family $S : A \to \N$ of natural numbers and the associated family
of maps $\degg(S) : A \to S^1 \to S^1$ sending $a$ to $\degg(S(a))$.
We write $M_S$ for the family of cofibers, which sends $a$ to the cofiber
of $\degg(S(a))$, a ``Moore space.''
We write $\susp{\degg(S)}$ for the family of suspensions of the maps $\degg(S(a))$,
and similarly consider $\suspsym^n M_S$ and $\suspsym^n \degg(S)$.

The main result of this section is:

\begin{thm}\label{theorem:localizationisnullification}
    Let $n\geq 1$.
    If $X$ is $n$-connected, then its
    $\degg(S)$-localization, its $\suspsym^{n-1} M_S$-nullification and
    its $\suspsym^n \degg(S)$-localization coincide.
\end{thm}

Note that this implies that the same is true for the other localizations ``between''
$\degg(S)$ and $\suspsym^n \degg(S)$.

\begin{proof}
By \cref{cor:preserve-n-connected},
$L_{\susp^n \degg(S)} X$ is $n$-connected,
since $\trunc{n}{\suspsym^n \, \sphere{1}} \eqvsym 1$.
Thus, by \cref{theorem:characterizinglocalness}, $L_{\susp^n \degg(S)} X$
is $\degg(S)$-local and $\suspsym^{n-1} M_S$-local.
So the natural maps
\[
  L_{\susp^n \degg(S)} X \lra L_{\suspsym^{n-1} M_S} X \lra L_{\degg(S)} X
\]
are equivalences, by \cref{lemma:comparelocalization}(3).
\end{proof}

This theorem fails without the connectedness hypothesis, as the next example shows.

\begin{eg}
    The type $\sphere{1}$ is not $\degg(k)$-local for $k > 1$ but it is $M_k$-null, by the following argument.
    On the one hand, $\loopspacesym \sphere{1}$ is equivalent to the integers, and the $k$-fold map
    is multiplication by $k$, which is not an equivalence.
    On the other, mapping from the cofiber sequence $\sphere{1}\to \sphere{1} \to M_k$ into $\sphere{1}$ we
    obtain the fiber sequence:
    \[
        (M_k \to_\ast \sphere{1}) \longhookrightarrow (\sphere{1} \to_\ast \sphere{1}) \xrightarrow{\;k\;} (\sphere{1}\to_\ast \sphere{1}),
    \]
    and the multiplication by $k$ map on the integers is injective, so the pointed mapping space $M_k \to_\ast \sphere{1}$ is contractible. Therefore $\sphere{1}$ is $M_k$-null by \cref{cor:pointed_null}.
\end{eg}

As an application of \cref{theorem:localizationisnullification}, we have:

\begin{cor}\label{corollary:commutativitylooplocalizationsimpconn}
    For a pointed, simply connected type $X$, we have that $\loopspacesym \eta : \loopspacesym X \to \loopspacesym L_{\degg(S)} X$
    is a $\degg(S)$-localization. In particular,
    $
        \loopspacesym L_{\degg(S)} X \simeq L_{\degg(S)} \loopspacesym X
    $.
\end{cor}

\begin{proof}
This follows from \cref{remark:commutativitylooplocalization} and \cref{theorem:localizationisnullification}.
\end{proof}

For $k > 1$, taking $X = K(\Z/k\Z, 1)$ and considering $L_{\degg(k)}$
shows that the assumption that $X$ is simply connected cannot be removed.

Note that we cannot necessarily iterate the interchange of $\loopspacesym$ and $\degg(S)$-localization,
since $\loopspacesym X$ might fail to be simply connected.

\medskip

It also follows that these localizations preserve $l$-connected types.

\begin{cor}\label{corollary:plocalizationpreservesconnectedness}
    For $l \geq -2$ and $n \geq 0$, if $X$ is $l$-connected, then $L_{\susp^n \degg(S)} X$ is $l$-connected.
\end{cor}

\begin{proof}
    If $n \geq l$, this follows from \cref{cor:preserve-n-connected}.
    If $n < l$, then \cref{theorem:localizationisnullification} implies
    that $L_{\susp^n \degg(S)} X \eqvsym L_{\susp^l \degg(S)} X$, putting us
    in the situation where $n = l$.
\end{proof}

The following proposition implies that localizations of groups always exist
(see \cref{rmk:localizationofgroups})
and is also used to prove \cref{theorem:localizationKgn}.

\begin{prop}\label{lemma:localizationlocalizesfirst}
    For $n \geq 1$,
    the $\degg(S)$-localization of a pointed, $(n-1)$-connected type $X$ localizes $\pi_n(X)$ away from $S$.
    The algebraic localization takes place in the category of groups if $n = 1$ and
    abelian groups otherwise.
\end{prop}

\begin{proof}
    By \cref{prop:homotopygroupsoflocalarelocal},
    we know that $\pi_n(L_{\degg(S)} X)$ is uniquely $S$-divisible.
    So it remains to show that precomposition with $\pi_n(\eta) : \pi_n(X) \to \pi_n(L_{\degg(S)} X)$
    induces an equivalence
    \[
        \precomp{\pi_n(\eta)} : \Hom(\pi_n(L_{\degg(S)} X),  H) \lra \Hom(\pi_n(X), H)
    \]
    for every uniquely $S$-divisible group $H$ (where $H$ is assumed to be abelian if $n>1$).
    Notice that this map is equivalent to 
    \[
        \precomp{\pi_n(\ttrunc{n}{\eta})} : \Hom(\pi_n(\ttrunc{n}{L_{\degg(S)} X}),  H) \lra \Hom(\pi_n(\ttrunc{n}{X}), H)
    \]
    which in turn is equivalent to
    \[
        \precomp{\ttrunc{n}{\eta}} : \left( \ttrunc{n}{L_{\degg(S)}X} \pto K(H, n) \right) \lra \left( \ttrunc{n}{X} \pto K(H, n) \right) 
    \]
    by \cref{theorem:catofgroups}, using the fact that $\ttrunc{n}{L_{\degg(S)}X}$ is still $(n-1)$-connected (\cref{corollary:plocalizationpreservesconnectedness}).
    Finally, this last map is equivalent to
    \[
        \precomp{\eta} : \left( L_{\degg(S)}X \pto K(H, n) \right) \lra \left( X \pto K(H, n) \right),
    \]
    since $K(H,n)$ is $n$-truncated.
    And this map is an equivalence, since $K(H,n)$ is $\degg(S)$-local by \cref{corollary:characterizationpdivisible}.
\end{proof}

The results of this section also let us deduce the following lemma,
which generalizes \cref{corollary:commutativitylooplocalizationsimpconn}
and will be used in \cref{ss:localization-of-homotopy-groups}.

\begin{lem}\label{lemma:lex}
    $\degg(S)$-localization preserves pullbacks of cospans of simply connected types.
    In particular, $\degg(S)$-localization preserves fiber sequences in which the base and total space are simply connected.
\end{lem}

\cref{example:nonlocalfib2,example:nonlocalfib} show that the $\degg(S)$-localization of a
fiber sequence is not a fiber sequence in general.

\begin{proof}
    Since the unit type is $\degg(S)$-local and simply connected,
    the second claim is a particular case of the first one.

    To prove the first claim, assume given a cospan of simply connected types $Y \rightarrow Z \leftarrow X$,
    and call its pullback $P$.
    We can first $\susp\degg(S)$-localize this cospan and then $\degg(S)$-localize it, to obtain pullbacks $Q'$ and $Q$,
    together with natural maps $P \to Q' \to Q$.
    Since $X$, $Y$, and $Z$ are simply connected, the map $Q' \to Q$ is an equivalence, by \cref{theorem:localizationisnullification}.
    On the other hand, using \cref{remark:preservationpullbacks} and setting $L\equiv L_{\degg(S)}$,
    we see that the natural map $P \to Q'$ is a $\degg(S)$-equivalence.
    To conclude the proof, notice that $Q$, being the limit of a diagram of $\degg(S)$-local types, is $\degg(S)$-local,
    so the map $P \to Q'$ is actually a $\degg(S)$-localization.
\end{proof}

We end this section with the following lemma, which is of a similar
flavor and which will be used in \cref{ss:localization-of-homotopy-groups}.

\begin{lem}\label{lemma:truncationpreserveslocal}
    For $l \geq -2$ and $n \geq 0$, the $l$-truncation of a $\suspsym^n \degg(S)$-local type is $\suspsym^n \degg(S)$-local.
\end{lem}

\begin{proof}
    Let $X$ be a $\suspsym^n \degg(S)$-local type.
    Fix $a : A$ and let $k = S(a)$.
    We must show that $k : \loopspacesym^{n+1} \ttrunc{l}{X} \to \loopspacesym^{n+1} \ttrunc{l}{X}$
    is an equivalence for each basepoint $x' : \ttrunc{l}{X}$.
    If $l - (n+1) \leq -2$, then $\loopspacesym^{n+1} \ttrunc{l}{X}$ is
    contractible, so this is clear.
    So assume that $l - (n+1) > -2$, and in particular that $l > -2$.
    Since being an equivalence is a mere proposition,
    we can assume that $x' = \tproj{l}{x}$ for some $x : X$.
    Recall that $l$-truncation
    is the subuniverse of separated types for $(l-1)$-truncation (\cref{example:truncationisseparated}).
    Applying \cref{Lseparatedloopspace} $n+1$ times (or \cite[Corollary~7.3.14]{hottbook})
    gives the equivalences
    in the square
\[
        \begin{tikzpicture}
          \matrix (m) [matrix of math nodes,row sep=2.5em,column sep=3em,minimum width=2em]
          { \loopspacesym^{n+1} \trunc{l}X & \trunc{l-n-1}{\loopspacesym^{n+1} X}\hspace*{-3ex} \\
            \loopspacesym^{n+1} \trunc{l} X & \trunc{l-n-1}{\loopspacesym^{n+1} X} . \hspace*{-3ex} \\};
          \path[->]
            (m-1-1) edge node [left] {$k$} (m-2-1)
            (m-2-2) edge node [above] {$\sim$} (m-2-1)
            (m-1-2) edge node [right]{$\trunc{l-n-1}{k}$} (m-2-2)
                    edge node [above] {$\sim$} (m-1-1)
            ;
        \end{tikzpicture}
    \]
    To show that the square commutes, it suffices to check this after precomposing
    with the truncation map $\loopspacesym^{n+1} X \to \trunc{l-n-1}{\loopspacesym^{n+1} X}$.
    And this follows since the equivalences commute with the natural maps from
    $\loopspacesym^{n+1} X$ and
    both vertical maps commute with $k : \loopspacesym^{n+1} X \to \loopspacesym^{n+1} X$.
    Since $X$ is $\suspsym^n \degg(k)$-local, the map on the right hand side is an equivalence,
    so the result follows.
\end{proof}

\subsection{Localization of Eilenberg--Mac Lane spaces}\label{ss:localizingKgn}

In this section, we compute the localization of a loop space as a mapping telescope,
in a way that is familiar from classical topology.
This specializes to give a concrete description of the
localization of an Eilenberg--Mac Lane space for an abelian group.
This is the key ingredient in the proofs of the results in the next section.

For the remainder of \cref{section:localizationaway}, we assume that our family
$S$ is indexed by the natural numbers, i.e., we have $S : \N \to \N$ and
$\degg(S)$ sends $i$ to $\degg(S(i))$.

For example, if we want to invert a decidable subset of the natural numbers
$T : \N \to \bool$, we can define a family $S : \N \to \N$ by
\[
    S(k) \defeq \begin{cases} k, & \text{if $T(k)$} \\
                              1, & \text{otherwise.}
                \end{cases}
\]
It is easy to see that a type is $\degg(S)$-local if and only if it is $\degg(k)$-local for every $k$ such that $T(k)$ is true.
The most important example of this form is the case of localizing \emph{at} a prime $p$,
in which we take $T$ to be the subset of primes different from $p$.

\begin{lem}\label{lemma:pmapisorthogonal}
    Given a pointed, simply connected type $X$, the $k$-fold map $k : \loopspacesym X \to \loopspacesym X$
    is an $L_{\degg(k)}$-equivalence.
\end{lem}

\begin{proof}
    We have the usual square
    \[
        \begin{tikzpicture}
          \matrix (m) [matrix of math nodes,row sep=2.5em,column sep=3em,minimum width=2em]
          { \loopspacesym X \strut & L_{\degg(k)} \loopspacesym X \strut \\
            \loopspacesym X \strut & L_{\degg(k)} \loopspacesym X . \strut \\};
          \path[->]
            (m-1-1) edge node [above] {$\eta$} (m-1-2)
                    edge node [left] {$k$} (m-2-1)
            (m-2-1) edge node [above] {$\eta$} (m-2-2)
            (m-1-2) edge node [right]{$L_{\degg(k)} k$} (m-2-2)
            ;
        \end{tikzpicture}
    \]
    Recall that the right vertical map is the unique map that makes the square commute.
    Now, by \cref{corollary:commutativitylooplocalizationsimpconn}, we know that
    the map $\loopspacesym \eta : \loopspacesym X \to \loopspacesym L_{\degg(k)} X$ is a $\degg(S)$-localization, so that,
    by the uniqueness of the right vertical map, the above square is equivalent to the square
    \[
        \begin{tikzpicture}
          \matrix (m) [matrix of math nodes,row sep=2.5em,column sep=3em,minimum width=2em]
          { \loopspacesym X \strut & \loopspacesym L_{\degg(k)} X \strut \\
            \loopspacesym X \strut & \loopspacesym L_{\degg(k)} X , \strut \\};
          \path[->]
            (m-1-1) edge node [above] {$\loopspacesym \eta$} (m-1-2)
                    edge node [left] {$k$} (m-2-1)
            (m-2-1) edge node [above] {$\loopspacesym \eta$} (m-2-2)
            (m-1-2) edge node [right]{$k$} (m-2-2)
            ;
        \end{tikzpicture}
    \]
    where the map on the right is the usual $k$-fold map.
    But in this square the right vertical map is an equivalence, since $L_{\degg(k)} X$
    is $\degg(k)$-local.
\end{proof}

We are almost ready to localize the loop space of a simply connected type away from $S : \N \to \N$.
Before doing this we need a general fact about sequential colimits.

\begin{rmk}\label{remark:equivalenceofseqcolim}
Given a sequential diagram $X$:
\[
    X_0 \xrightarrow{\,h_0\,} X_1 \xrightarrow{\,h_1\,} \cdots
\]
we can consider the shifted sequential diagram $X_1 \to X_2 \to \cdots$, which we denote $X[1]$.
Analogously, we can shift a natural transformation $f : X \to Y$ between sequential
diagrams to obtain a natural transformation $f[1] : X[1] \to Y[1]$.
Notice that the transition maps $h_n$ induce a natural transformation $h : X \to X[1]$, and that moreover, the
induced map $\colim\, h : \colim X \to \colim X[1]$ is an equivalence.
It then follows that if $f : X \to Y$ is a natural transformation between sequential diagrams, such that there exist $g, g' : Y \to X[1]$
with $g \circ f = h$ and $f[1] \circ g' = i$, where $i : Y \to Y[1]$ is the transition map of $Y$, then $f$ induces an equivalence
$\colim\, f : \colim X \xrightarrow{\sim} \colim Y$.
\end{rmk}

\begin{thm}\label{theorem:localizationastelescope}
    Let $X$ be pointed and simply connected, and
    define $s : \N \to \N$ by $s(n) = \prod_{i=0}^n S(i)$.
    Then the $\degg(S)$-localization of $\loopspacesym X$ is equivalent to the colimit of the sequence
    \[\loopspacesym X \xrightarrow{s(0)}\loopspacesym X \xrightarrow{s(1)} \loopspacesym X \xrightarrow{s(2)} \cdots, \]
    where as usual a natural number $k$ is used to denote the $k$-fold map.
\end{thm}

\begin{proof}
    We must show that the colimit is $\degg(S)$-local and that
    the map $\loopspacesym X \to \colim\, \loopspacesym X$ is an $L_{\degg(S)}$-equivalence.

    To prove that $\loopspacesym X \to \colim\, \loopspacesym X$ is an $L_{\degg(S)}$-equivalence
    we use \cref{lemma:orthogonalcomposition}, so it is enough to show that all the maps in the
    diagram are $L_{\degg(S)}$-equivalences. To prove this last fact notice that \cref{lemma:pmapisorthogonal}
    implies that the $s(k)$-fold map is an $L_{\degg(s(k))}$-equivalence. But in particular this implies that it is also an
    $L_{\degg(S)}$-equivalence, since every $\degg(S)$-local type is $\degg(s(k))$-local for every $k$.

    Now we must show that the colimit is $\degg(S(k))$-local for each $k$.
    So fix $k$ and recall that
    the colimit is equivalent to the colimit of the sequence starting at $k$.
    That is, we consider the colimit of the sequential diagram with objects
    $C_n = \loopspacesym X$ and transition maps $h_n : C_n \to C_{n+1}$ given by $s(n+k)$.
    Then $\colim_n C_n$ is pointed and connected by \cite{DoornRijkeSojakova},
    so it is enough to check that the map $S(k) : \loopspacesym \colim_n C_n \to \loopspacesym \colim_n C_n$ is an equivalence.
    By the commutativity of loop spaces and sequential colimits \cite{DoornRijkeSojakova},
    it is then enough to show that the natural transformation $S(k) : \loopspacesym C_n \to \loopspacesym C_n$
    induces an equivalence in the colimit.

    By the Eckmann-Hilton argument~\cite[Theorem~2.1.6]{hottbook}, the transition map
    $\loopspacesym h_n : \loopspacesym C_n \to \loopspacesym C_{n+1}$ is homotopic to $s(n+k)$.
    To finish the proof, we apply \cref{remark:equivalenceofseqcolim} to the sequential
    diagram $\loopspacesym C_0 \to \loopspacesym C_1 \to \cdots$,
    taking $f$ to be the natural transformation given by
    $S(k)$ in every degree, and both $g$ and $g'$ to be the natural transformation given by
    $s(n + k)/S(k) : \loopspacesym C_n \to \loopspacesym C_{n+1}$ in degree $n$.
    We have that $g_n \circ f_n = f_{n+1} \circ g'_n = s(n + k)$, and moreover that
    $g \circ f = f[1] \circ g' = h$ as natural transformations.
    So $S(k) : \loopspacesym C_n \to \loopspacesym C_n$ induces an equivalence in the colimit,
    as needed.
\end{proof}

Using the fact that $K(G,n)$ is the loop space of a simply connected space
when $G$ is abelian we deduce:

\begin{cor}\label{corollary:localizationKgn}
    For $n \geq 1$ and $G$ abelian,
    the $\degg(S)$-localization of $K(G,n)$ is equivalent to the colimit of the sequence
    \[
      K(G,n) \xrightarrow{s(0)} K(G,n) \xrightarrow{s(1)} \cdots .  \tag*{\qed}
    \]
\end{cor}

In particular, it follows from \cite{DoornRijkeSojakova} that the
$\deg(S)$-localization of $K(G,n)$ is $n$-truncated and $(n-1)$-connected,
so it is again a $K(G',n)$ for some group $G'$.
By \cref{lemma:localizationlocalizesfirst}, $G'$ is the algebraic localization of $G$
away from $S$ (in the category of groups if $n = 1$ and abelian groups otherwise).
We deduce:

\begin{thm}\label{theorem:localizationKgn}
    Let $n\geq 1$ and let $G$ be any abelian group.
    Then the $\degg(S)$-localization map $K(G,n) \to L_{\degg(S)} K(G,n)$ is a map between Eilenberg--Mac Lane spaces and
    is induced by an algebraic localization of $G$ away from $S$.
    Moreover, the algebraic localization of $G$ is equivalent to the colimit
    \[
      G \xrightarrow{s(0)} G \xrightarrow{s(1)} \cdots
    \]
    and is therefore abelian, even when $n = 1$.  \qed
\end{thm}

It follows that the two types of algebraic localization agree for abelian groups,
so we do not need to distinguish between them in what follows.

\subsection{Localization of homotopy groups}\label{ss:localization-of-homotopy-groups}

In this section we prove the main theorem of the paper:
for simply connected types, $\degg(S)$-localization localizes all homotopy groups away from $S$.
We continue to assume that our family $S$ is indexed by the natural numbers,
as the essential ingredient is our result on the localization of Eilenberg--Mac Lane spaces from the previous section.
The other ingredient we need is that for simply connected types,
truncation commutes with $\degg(S)$-localization, which follows
from the next lemma.

\begin{lem}\label{lemma:locofscntrunc}
    The $\degg(S)$-localization of a simply connected, $n$-truncated type is $n$-truncated.
\end{lem}

\begin{proof}
    First, we show that it suffices to prove the statement for pointed types.
    Let us assume given a simply connected and $n$-truncated type $X$, and let us denote the statement
    of the theorem by $P(X)$. If we prove $X \to P(X)$ it follows that $\ttrunc{-1}{X} \to P(X)$,
    since being $n$-truncated is a mere proposition. But this is enough, for $X$ is
    simply connected, which implies that $\ttrunc{1}{X}$ is contractible, and hence that
    $\ttrunc{-1}{X}$ is inhabited.

    We now assume that $X$ is pointed and proceed by induction.
    If $X$ is $-2$, $-1$, $0$ or $1$-truncated, we are done, since $X$ is also simply connected, and thus contractible.
    If $X$ is $(n+1)$-truncated and $n > 0$, consider the fiber sequence
    \[ K(G,n+1) \longhookrightarrow X \lra \ttrunc{n}{X}. \]
    Since the types in the fiber sequence are simply connected, $\degg(S)$-localization
    preserves this fiber sequence, by \cref{lemma:lex}. So we obtain a fiber sequence
    \[L_{\degg(S)} K(G,n+1) \longhookrightarrow L_{\degg(S)}X \lra L_{\degg(S)} \ttrunc{n}{X}. \]
    The type $L_{\degg(S)} \ttrunc{n}{X}$ is $n$-truncated by the induction hypothesis
    and $L_{\degg(S)} K(G,n+1)$ is $(n+1)$-truncated by \cref{theorem:localizationKgn}.
    So $L_{\degg(S)} X$ must be $(n+1)$-truncated as well.
\end{proof}

The proof of \cref{lemma:locofscntrunc} shows that we can compute
the $\degg(S)$-localization of a simply connected $n$-type $X$ by
localizing the Postnikov tower of $X$.

From \cref{lemma:locofscntrunc}, \cref{lemma:truncationpreserveslocal} and
\cref{lemma:commutelocalization}, we deduce:

\begin{cor}\label{corollary:localizationandtruncationcommute}
    For simply connected types, $\degg(S)$-localization and $n$-trunca\-tion commute.\qed
\end{cor}

We now give the main theorem of the paper.

\begin{thm}\label{theorem:localizationlocalizes}
    The $\degg(S)$-localization of a pointed, simply connected type $X$ localizes all of the
    homotopy groups away from $S$.
\end{thm}

\begin{proof}
    Since $\degg(S)$-localization preserves the property of being simply connected, it is immediate that $\pi_1(L_{\degg(S)} X)$ is trivial.
    Now fix $n \geq 1$ and consider the fiber sequence
    \[
        K(\pi_{n+1}(X),n+1) \longhookrightarrow \ttrunc{n+1}{X} \lra \ttrunc{n}{X}.
    \]
    Notice that all of the types in the fiber sequence are simply connected.
    Applying $\degg(S)$-localization we obtain a map of fiber sequences
    \[
        \begin{tikzpicture}
          \matrix (m) [matrix of math nodes,row sep=2em,column sep=3em,minimum width=2em]
            { K(\pi_{n+1}(X),n+1) \strut & \ttrunc{n+1}{X} \strut & \ttrunc{n}{X} \strut \\
            L_{\degg(S)} K(\pi_{n+1}(X),n+1) &  \ttrunc{n+1}{L_{\degg(S)}X} & \ttrunc{n}{L_{\degg(S)} X}\\};
          \path[->]
            (m-1-1) edge [right hook->] node [right] {} (m-1-2)
                    edge node [right] {} (m-2-1)
            (m-1-2) edge node [right] {} (m-1-3)
                    edge node [right] {} (m-2-2)
            (m-2-2) edge node [right] {} (m-2-3)
            (m-2-1) edge [right hook->] node [right] {} (m-2-2)
            (m-1-3) edge node [right] {} (m-2-3)
            ;
        \end{tikzpicture}
    \]
    by \cref{lemma:lex} combined with the commutativity of truncation and
    $\degg(S)$-localization (\cref{corollary:localizationandtruncationcommute}).
    The result now follows from \cref{theorem:localizationKgn}.
\end{proof}

We also have a partial converse to \cref{theorem:localizationlocalizes}.

\begin{thm}\label{theorem:characterize-localization}
  Let $X$ and $X'$ be pointed, simply connected $n$-types, for some $n \geq 0$.
  If $f : X \to X'$ is a pointed map such that $\pi_m(f) : \pi_m(X) \to \pi_m(X')$ is
  an algebraic localization away from $S$ for each $m > 1$,
  then $f$ is a $\degg(S)$-localization of $X$.
\end{thm}

\begin{proof}
  By \cref{prop:homotopygroupsoflocalarelocal}, $X'$ is $\degg(S)$-local.
  Therefore, we have a commuting triangle
  \[
    \begin{tikzcd}
      X \arrow[r,"f"] \arrow[d,swap,"\eta"] & X' \\
      L_{\degg(S)} X \arrow[ur,dashed,swap,"f'"]
    \end{tikzcd}
  \]
  and it suffices to show that $f'$ is an equivalence.
  Since $L_{\degg(S)} X$ is also a pointed, simply connected $n$-type,
  if we show that $\pi_m(f')$ is a bijection for each $m > 1$,
  the truncated Whitehead theorem~\cite[Theorem~8.8.3]{hottbook} implies
  that $f'$ is an equivalence.
  By assumption, $\pi_m(f)$ is an algebraic localization of $\pi_m(X)$.
  By \cref{theorem:localizationlocalizes}, the same is true of $\pi_m(\eta)$.
  Since $\pi_m(f') \circ \pi_m(\eta) = \pi_m(f)$, it follows that
  $\pi_m(f')$ is a bijection.
\end{proof}

\subsection{Algebraic localization of abelian groups}\label{ss:abelian}

As observed in \cref{rmk:localizationofgroups},
\cref{theorem:localizationKgn} implies that the localization of an abelian group
away from a set $S$, in the category of groups, coincides with its localization in
the category of abelian groups. It moreover provides a construction of the localization.
Since \cref{theorem:localizationKgn}
has an indirect, homotopical proof, in this section we give a short, independent, algebraic proof of these results.

The following elementary lemma is the key ingredient, and its proof is very similar to the
proof of~\cite[Lemma~5.4.5]{MayPonto}.
 
\begin{lem}
    Let $H$ be a group, let $n,m : \N$, and let $x$, $y$, $\hat{x}$ and $\hat{y}$ be
    elements of $H$ with the property that
    $\hat{x}$ is the unique element of $H$ such that $\hat{x}^n = x$ and
    $\hat{y}$ is the unique element of $H$ such that $\hat{y}^m = y$.
    If $x$ and $y$ commute, then $\hat{x}$ and $\hat{y}$ commute.
\end{lem}

\begin{proof}
    We first show that $x$ and $\hat{y}$ commute.
    The $m$-th power of $x \hat{y} x^{-1}$ is $x \hat{y}^m x^{-1} = x y x^{-1} = y$,
    so by the uniqueness of $\hat{y}$, it must be the case that $x \hat{y} x^{-1} = \hat{y}$.

    Similarly, we have that
    the $n$-th power of $\hat{y} \hat{x} \hat{y}^{-1}$ is
    $\hat{y} \hat{x}^n \hat{y}^{-1} = \hat{y} x \hat{y}^{-1} = x$,
    so by the uniqueness of $\hat{x}$, we have $\hat{y} \hat{x} \hat{y}^{-1} = \hat{x}$.
\end{proof}

\begin{prop}
    Let $G$ be an abelian group, let $n : \N$, and let $n : G \to G$ be the multiplication by $n$ map.
    Then, for every uniquely $n$-divisible group $H$, the precomposition map
    \[
        n^* : \Hom(G,H) \lra \Hom(G,H)
    \]
    is an equivalence.
\end{prop}

\begin{proof}
    Given a map $f : G \to H$, we have to show that there exists a unique $\hat{f} : G \to H$ such that
    $\hat{f} \circ n = f$.
    Since $H$ is uniquely $n$-divisible, we have an inverse function $\phi : H \to H$ to the $n$-th power map $n : H \to H$.
    It follows that if $\hat{f}$ exists, then $\hat{f} = \phi \circ f$.
    So we only have to check that $\hat{f}$ is a group homomorphism.

    It is clear that $\hat{f}$ preserves the identity element, so it remains to show that it preserves
    the group operation. Take two elements $a,b : G$, and consider $f(a),f(b) : H$. These are two commuting
    elements that have unique $n$-th roots.
    So their $n$-th roots $\hat{f}(a) = \phi(f(a))$ and $\hat{f}(b) = \phi(f(b))$ must also
    commute. This implies that
    \[
        (\hat{f}(a) \cdot \hat{f}(b))^n = \hat{f}(a)^n \cdot \hat{f}(b)^n = f(a) \cdot f(b) = f(a + b) .
    \]
    So, by the uniqueness of $n$-th roots, $\hat{f}(a) \cdot \hat{f}(b) = \phi(f(a + b)) = \hat{f}(a+b)$.
\end{proof}
\enlargethispage{10pt}

Using this result, given an abelian group $G$,
it is straightforward to prove that the colimit of the sequence displayed in
\cref{theorem:localizationKgn} is a localization of $G$ away from the family $S$
in the category of groups.
Moreover, this colimit is abelian, so it is also the localization of $G$
in the category of abelian groups.

\printbibliography

\end{document}